\documentclass[a4paper,11pt]{article}

\usepackage{amsmath}
\usepackage{amsfonts}
\usepackage{amssymb}
\usepackage{amstext    }
\usepackage{amsthm    }

\usepackage[a4paper]{geometry}
\usepackage{mathrsfs}

\usepackage{hyperref}

\usepackage{color}

\usepackage[T1]{fontenc}
\usepackage{graphicx}
\usepackage[all,cmtip]{xy}

\usepackage{epsf}
\usepackage{psfrag}

\usepackage{fancyhdr}

\usepackage{url}

\newtheorem{theorem}{Theorem}[section]
\newtheorem{lemma}[theorem]{Lemma}
\newtheorem{definition}[theorem]{Definition}
\newtheorem{proposition}[theorem]{Proposition}
\newtheorem{remark}[theorem]{Remark}
\newtheorem{corollary}[theorem]{Corollary}

\newcommand{\cali}[1]{\mathscr{#1}}

\numberwithin{equation}{section}

\newcommand{\supp}{{\rm supp}}

\newcommand{\bif}{{\rm Bif}}

\newcommand{\Hc}{\cali{H}}

\newcommand{\Jc}{\cali{J}}

\newcommand{\Pc}{\cali{P}}

\newcommand{\Uc}{\cali{U}}

\newcommand{\Db}{\mathbb{D}}
\newcommand{\Pb}{\mathbb{P}}

\newcommand{\Cb}{\mathbb{C}}
\newcommand{\C}{\mathbb{C}}
\newcommand{\Nb}{\mathbb{N}}

\newcommand{\Zb}{\mathbb{Z}}

\newcommand{\Rb}{\mathbb{R}}

\usepackage[T1]{fontenc}
\usepackage[english]{babel}

\author{Johan Taflin}

\title{Blenders near polynomial product maps of $\Cb^2$}
\begin{document}
\maketitle

\begin{abstract}
In this paper we show that if $p$ is a bifurcating polynomial then the product map $(z,w)\mapsto(p(z),q(w))$ can be approximated by polynomial skew products possessing special dynamical objects called \textit{blenders}. Moreover, these objects can be chosen to be of two types: \textit{repelling} or \textit{saddle}. As a consequence, such a product map belongs to the closure of the interior of two different sets: the bifurcation locus of $\Hc_d(\Pb^2)$ and the set of endomorphisms having an attracting set of non-empty interior.

Similar techniques also give the first example of an attractor with non-empty interior or of a saddle hyperbolic set which is robustly contained in the small Julia set and whose unstable manifolds are all dense in $\Pb^2.$

In an independent part, we use perturbations of Hénon maps to obtain examples of attracting sets containing repelling points and also of quasi-attractors which are not attracting sets.
\end{abstract}
\section{Introduction}
The idea of \textit{blenders} was introduced by Bonatti and D\'iaz in \cite{bonatti-diaz-blender} to obtain robustly transitive diffeomorphisms which are not hyperbolic. Since then, blenders have become an important tool in smooth dynamics especially to build examples exhibiting new phenomena. Although there is no precise definition of a blender $\Lambda$ its main properties are that it persists under small $C^1$ perturbations and its stable set $W_\Lambda^s$ (or unstable set) intersects an open family of submanifolds of codimension strictly smaller than the topological dimension of $W_\Lambda^s,$ i.e. these robust intersections don't have a topological origin. We refer to \cite{beyond}, \cite{berger-crovisier-pujals-blender} and \cite{whatis-blender} for introductions to the subject.

Recently, blenders (called of repelling type in what follows) were introduced in complex dynamics by Dujardin \cite{dujardin-bif} (see also \cite{biebler-blender}) in order to prove that the bifurcation locus as defined in \cite{bbd-bif} has non-empty interior in the family $\Hc_d(\Pb^k)$ of holomorphic endomorphisms of degree $d$ of $\Pb^k$ for all $d\geq2$ and $k\geq2,$ i.e. there exist robust bifurcations in $\Hc_d(\Pb^k).$ This contrasts with the case of one complex variable since the classical results of Ma\~n\'e, Sad and Sullivan \cite{mss-rational} and Lyubich \cite{lyubich-bif} imply that the bifurcation locus has empty interior in any family of rational maps of $\Pb^1,$ i.e. the stable set is always an open and dense subset.

To obtain such blenders, Dujardin considers some perturbations of product maps of $\Cb^2$ of the form $(z,w)\mapsto(p(z),q(w))$ where $q(w)=w^d+\kappa$ with $\kappa$ large and $p$ belonging to a specific subset of the bifurcation locus $\bif(\Pc_d)$ of the family $\Pc_d$ of all polynomials of degree $d.$ Our main result is that these assumptions can be reduced to $p\in\bif(\Pc_d),$ i.e. blenders always exist near bifurcations of product maps.
\begin{theorem}\label{th-main}
Let $d\geq2.$ If $p$ and $q$ are two elements of $\Pc_d$ such that $p\in\bif(\Pc_d)$ then the map $(p,q)\in\Hc_d(\Pb^2)$ can be approximated by polynomial skew products having an iterate with a blender of repelling type (resp. of saddle type).
\end{theorem}
We refer to Section \ref{sec-blender} for details about these two closely related notions of blenders. In this introduction, we just point out that in our definition a blender of repelling type (resp. saddle type) is contained in a repelling (resp. saddle) invariant hyperbolic set.

The main arguments are the following. First, we observe that some perturbations of the map $(z,w)\mapsto(z,q^l(w))$ have a blender if $l\geq1$ is large enough (see Proposition \ref{prop-exi-rep} and Proposition \ref{prop-exi-selle}). Then, using the fact that $p\in\bif(\Pc_d)$ can be approximated by polynomials with a parabolic cycle we show that the above perturbations can be realized as the $l$-th iterate of perturbations of $(z,w)\mapsto(p(z),q(w))$ (see Theorem \ref{th-constru}). As a consequence of this construction in the repelling case, we give a positive answer to a conjecture of Dujardin \cite[Conjecture 5.6]{dujardin-bif}.
\begin{corollary}\label{coro-bif}
The bifurcation locus of the family $\Pc_d\times\Pc_d$ is contained in the closure of the interior of the bifurcation locus in $\Hc_d(\Pb^2).$
\end{corollary}
On the other hand, blenders of saddle type give rise to invariant sets with non empty interior. Actually, their unstable set has non empty interior. As a consequence, we obtain the following result.
\begin{corollary}\label{coro-as}
The bifurcation locus of the family $\Pc_d\times\Pc_d$ is contained in the closure of the interior of the set of maps in $\Hc_d(\Pb^2)$ possessing a proper attracting set with non-empty interior.
\end{corollary}
Recall that an \textit{attracting set} $A$ of an endomorphism $f$ is a compact set of the form $A=\cap_{n\geq0}f^n(U)$ where $U$ is an open set (called a \textit{trapping region}) such that $\overline{f(U)}\subset U.$ Such a subset of $\Pb^2$ is called \textit{proper} if $A\neq\Pb^2.$

To the best of our knowledge, no previous example of proper attracting set with non-empty interior was known. The result above says that they are abundant. Moreover, these attracting sets are ``near to collapse''. An interesting question is to understand which part of the bifurcation locus in $\Hc_d(\Pb^2)$ can be approximated by maps with such ``collapsing'' attracting sets and whether or not these sets have non-empty interior (when they are not finite). This can be related to the Lyapunov exponents of the equilibrium measures on attracting sets introduced in \cite{d-attractor} and \cite{t-attractor}. Observe that in one variable the full bifurcation locus can be characterized by the fact that the number of attracting cycles is not constant in any neighborhood of a bifurcating parameter.

The technique to obtain Corollary \ref{coro-as} is elementary and flexible. It is possible to adapt it to produce examples with additional properties. In particular, recall that an \textit{attractor} is an attracting set which is topologically transitive. The following result gives an example of attractor whose interior is uniformly wide in some sense.
\begin{theorem}\label{th-att}
There exists an endomorphism of $\Pb^2$ which has a proper attractor containing an algebraic curve in its interior.
\end{theorem}
An alternative way to obtain attracting sets with non empty interior is to perturb (compositions of) Hénon maps. Indeed, an easy observation, which apparently is not present in the literature, is that all the dynamics of a Hénon map takes place in a proper attracting set of $\Pb^2.$ Hence, small perturbations give endomorphisms of $\Pb^2$ with an attracting set which inherits several properties of the original Hénon map (see also Remark \ref{rk-henon}). Although these examples cannot be transitive, they exhibit other interesting phenomena. The first one is simply the existence of attracting set possessing repelling cycles. Another issue is about quasi-attractors. A \textit{quasi-attractor} is an infinite decreasing intersection of attracting sets. These objects play an important role in dynamics and they have been studied in the complex setting by \cite{fw-attractor} and  \cite{t-attractor}.
\begin{theorem}\label{th-pertu}
There exist proper attracting sets of $\Pb^2$ with infinitely many repelling cycles. Moreover, there exist quasi-attractors which are not attracting sets.
\end{theorem}
Actually, these examples lead to endomorphisms of $\Pb^2$ with uncountably many quasi-attractors (see Remark \ref{rk-uncountable}). Notice that in \cite{t-attractor} it was shown that a holomorphic endomorphism of $\Pb^k$ has at most countably many quasi-attractors which are minimal (with respect to the inclusion). The same paper established that if $(A_n)_{n\geq1}$ is a decreasing sequence of attracting sets in $\Pb^2$ such that $A:=\cap_{n\geq1}A_n$ is not an attracting set then the Hausdorff dimension of each $A_n$ has to be greater or equal to $3.$ In the result above, the sets $A_n$ have non-empty interior and thus have maximal Hausdorff dimension $4.$

Although recent developments partially filled this lack (see e.g. \cite{dujardin-non-laminar}, \cite{wandering}, \cite{bianchi-t-desboves}), complex dynamics in several variables misses well-understood and interesting examples. It is likely that, as in smooth dynamics, the study of all the different types of blenders in $\Pb^k$ can help to increase this set of examples and could also become a general mechanism which explains some specific phenomena. In that direction, our last result below gives an open set of endomorphisms of $\Pb^2$ in which each map displays unusual phenomena. The mechanism behind it is a robust ``heterodimensional cycle'' between two blenders, one of repelling type and one of saddle type.
\begin{theorem}\label{th-cycle}
There exist $d\geq2$ and an open set $\Omega\subset\Hc_d(\Pb^2)$ which contains skew products such that $\Omega\subset\bif(\Hc_d(\Pb^2))$ and each $f$ in $\Omega$ possesses a hyperbolic basic set of saddle type $\Lambda$ with positive entropy whose unstable set $W^u_{\Lambda}$ is a Zariski open set of $\Pb^2$ and 
\begin{itemize}
\item for each $\widehat x$ in the natural extension of $\Lambda,$ the unstable manifold $W^u_{\widehat x}$ is dense in $\Pb^2,$
\item for each $x\in\Lambda,$ the stable manifold $W^s_x$ is contained in the small Julia set $\Jc_2(f).$
\end{itemize}
In particular, we have $\Lambda\subset\Jc_2(f).$
\end{theorem}
The small Julia set is by definition the support of the so-called \textit{equilibrium measure} which is of repelling nature (all its Lyapunov exponents are positive \cite{briend-duval-expo}). The above statement in the case of repelling hyperbolic sets is classical but this result is the first example of a saddle hyperbolic set which is robustly contained in the small Julia set. It also provides the first example of an endomorphism of $\Pb^2$ with a saddle point whose unstable manifold is dense (and moreover in a robust way in $\Hc_d(\Pb^2)$).

We have been aware of \cite{dujardin-bif} while we were working on this subject. Our general approach is similar, although mostly done independently. For this reason, we will use several results from \cite{dujardin-bif} without proving them (see e.g. Section \ref{subsec-bif}).  We also follow Dujardin's strategy in order to prove Corollary \ref{coro-bif} since our original approach to obtain open sets of bifurcation (see Remark \ref{rk-cycle}) was a priori insufficient. However, the construction of blenders presented here is self-contained, our approach is less precise than the one in \cite{dujardin-bif} but also more flexible. The paper is organized as follows. In Section \ref{sec-back}, we give the necessary background for the sequel about hyperbolic sets and bifurcation theory. Then, Section \ref{sec-blender} is devoted to the construction of the two different types of blenders in $\Pb^2$ and in Section \ref{sec-cons} we prove that these objects exist near bifurcations of product maps. In Section \ref{sec-ouvert-bif} and Section \ref{sec-att} we establish Corollary \ref{coro-bif} and Corollary \ref{coro-as} respectively using blenders of repelling type and of saddle type. In the latter, we also show how one can adapt the idea of saddle blenders to obtain attractors with non-empty interior. In Section \ref{sec-cycle}, we give examples of maps with a cycle between a repelling blender and a saddle blender and we show how the existence of such objects easily imply Theorem \ref{th-cycle}. Finally, in Section \ref{sec-henon} we consider perturbations of Hénon maps in order to obtain Theorem \ref{th-pertu}.

\paragraph*{Acknowledgements.} I would like to thank Christian Bonatti for introducing me to blenders and for many interesting discussions on dynamics.

\section{Background}\label{sec-back}
We refer to \cite{ds-lec} for a detailed introduction to complex dynamics in several variables. To an endomorphism $f$ of $\Pb^k$ of degree $d\geq2$ it is possible to associate different invariant objects. The more classical one is the Julia set of $f.$ More generally, Forn\ae ss and Sibony \cite{fs-cdhd2} associated to $f$ its \textit{Green current} $T$ and for $1\leq p\leq k$ the \textit{Julia set of order} $p$, $\Jc_p(f)$ is the support of the self-power $T^p.$ They form a filtration $\Jc_k(f)\subset\cdots\subset\Jc_1(f)$ of totally invariant sets for $f.$ The smallest one $\Jc_k(f),$ sometime called the \textit{small Julia set}, is the support of the equilibrium measure $\mu:=T^k$ of $f.$ By \cite{briend-duval-expo}, the repelling cycles are dense in $\Jc_k(f)$ and if $k=2$ and $f$ is induced by a skew product of $\Cb^2$ then it follows from \cite{jonsson-skew} that $\Jc_2(f)$ is exactly the closure of the repelling cycles.

\subsection{Hyperbolic sets}
Another type of interesting invariant objects is given by hyperbolic sets. Since an endomorphism $f$ of $\Pb^k$ is non-invertible, the definition of hyperbolicity involves the \textit{natural extension} of $\Pb^k$ given by $\widehat\Pb^k:=\{(x_i)_{i\leq0}\in(\Pb^k)^{\Zb_{\leq0}}\,|\, f(x_i)=x_{i+1}\}$ (see e.g. \cite[Paper I]{jonsson-phd} for a detailed exposition on the subject). There is a natural projection $\pi\colon\widehat\Pb^k\to\Pb^k$ defined by $\pi((x_i))=x_0$ and there is a unique homeomorphism $\widehat f$ of $\widehat\Pb^k$ which satisfies $f\circ\pi=\pi\circ\widehat f.$ One can use the projection $\pi$ to lift the tangent bundle of $\Pb^k$ to a bundle $T_{\widehat\Pb^k}$ on which the derivative $Df$ of $f$ acts naturally. And one says that a compact invariant set $\Lambda\subset\Pb^k$ is a \textit{hyperbolic set} if the restriction of this bundle to $\widehat\Lambda:=\{(x_i)_{i\leq0}\in\widehat{\Pb}^k\,|\, x_i\in\Lambda \text{ for all }i\leq0\}$ admits a continuous splitting $E^s\oplus E^u,$ invariant by $Df$ and such that there exist constants $C>0$ and $0<\lambda<1$ with $\|Df^nu\|\leq C\lambda^n\|u\|$ and $\|(Df^n)^{-1}v\|\leq C\lambda^n\|v\|$ for all $(u,v)\in E^s\times E^u.$ A key point about hyperbolic set is that for all $\widehat x=(x_i)_{i\leq0}$ in $\widehat\Lambda$ there exist an unstable manifold $W^u_{\widehat x}$ and a stable manifold $W^s_{x_0}.$ And we define the \textit{unstable set} of $\Lambda$ by $W_\Lambda^u:=\cup_{\widehat x\in\widehat\Lambda}W_{\widehat x}^u.$

A hyperbolic set $\Lambda$ is said to be a \textit{basic set} if $f_{|\Lambda}$ is transitive and if there exists a neighborhood $\widehat U$ of $\widehat\Lambda$ such that $\widehat\Lambda=\cap_{n\in\Zb}\widehat f^n(\widehat U)$ (i.e. $\widehat\Lambda$ is \textit{locally maximal}). Such a set satisfies the shadowing lemma \cite[Theorem 2.4]{jonsson-phd} so the periodic points are dense in $\Lambda.$ Another important point about basic sets is that they are structurally stable in the following sense (see Proposition 1.4 and Corollary 2.6 of \cite{jonsson-phd}). If $g$ is $C^1$-close to $f$ and $\widehat U$ is a sufficiently small neighborhood of $\widehat\Lambda$ then $\widehat g$ is conjugated to $\widehat f_{\widehat\Lambda}$ on the set $\widehat\Lambda_g:=\cap_{n\in\Zb}\widehat g(\widehat U)$ which projects to a basic set $\Lambda_g$ of $g.$ And if $\Lambda$ is repelling (i.e. $E^u=T_{\widehat\Lambda}$) then the conjugation is not only defined in the natural extension but directly between $\Lambda$ and $\Lambda_g.$ Moreover, in the holomorphic case these sets can be followed by a holomorphic motion (see \cite[Theorem A.4]{bbd-bif}).
\begin{theorem}\label{th-motion}
Let $(f_\lambda)_{\lambda\in M}$ be a holomorphic family of endomorphisms such that $f_{\lambda_0}$ admits a basic repeller $\Lambda_{\lambda_0}$ for some $\lambda_0\in M.$ Then there exist a neighborhood $U$ of $\Lambda_{f_0},$ a neighborhood $B\subset M$ of $\lambda_0$ and a continuous map $h\colon B\times\Lambda_{\lambda_0}\to U$ such that
\begin{itemize}
\item[(i)] $\lambda\mapsto h_\lambda(x)$ is holomorphic on $B$ with $h_{\lambda_0}(x)=x$ for every $x\in\Lambda_{\lambda_0},$
\item[(ii)] $x\mapsto h_\lambda(x)$ is injective on $\Lambda_{\lambda_0}$ for every $\lambda\in B,$
\item[(iii)] $h_\lambda\circ f_{\lambda_0}=f_\lambda\circ h_\lambda$ on $\Lambda_{\lambda_0}$ for every $\lambda\in B.$
\end{itemize}
Moreover, the set $\Lambda_\lambda:=h_\lambda(\Lambda_{\lambda_0})$ is a basic repeller for $f_\lambda$ satisfying $\Lambda_\lambda=\bigcap_{n\geq0}f^{-n}_\lambda(U).$
\end{theorem}

\subsection{Bifurcations and basic repellers}\label{subsec-bif}
A bifurcation theory for the small Julia set $\Jc_k(f)$ was recently developed by Berteloot-Bianchi-Dupont \cite{bbd-bif} in $\Pb^k$ for $k\geq2.$ A holomorphic family of endomorphisms of $\Pb^k$ parametrized by a complex manifold $M$ is given by a holomorphic map $f\colon M\times\Pb^k\to\Pb^k.$ We denote such a family $(f_\lambda)_{\lambda\in M}$ where $f_\lambda(x):=f(\lambda,x).$ For each $\lambda\in M$ we can consider the small Julia set $\Jc_k(f_\lambda)$ of $f_\lambda,$ its critical set $C(f_\lambda)$ and its \textit{postcritical set}
$$P(f_\lambda):=\bigcup_{n\geq1}f_\lambda^n(C(f_\lambda)).$$

In \cite{bbd-bif}, the authors define a special closed subset $\bif(M)$ of $M$ called the \textit{bifurcation locus} of the family $(f_\lambda)_{\lambda\in M}.$ They obtain several characterizations of $\bif(M).$ In the sequel we will use the following one based on the notion of Misiurewicz parameters. A parameter $\lambda_0\in M$ is called a \textit{Misiurewicz parameter} if $f_{\lambda_0}$ admits a repelling periodic point $x_{\lambda_0}$ in $\Jc_k(f_{\lambda_0})\cap f_{\lambda_0}^{n_0}(C(f_{\lambda_0}))$ for some $n_0\geq0,$ whose holomorphic continuation as repelling point $x_\lambda$ is outside $f_\lambda^{n_0}(C(f_\lambda))$ for $\lambda\in M$ arbitrarily close to $\lambda_0.$
\begin{theorem}[\cite{bbd-bif}]\label{th-misiu}
If $(f_\lambda)_{\lambda\in M}$ is a holomorphic family of endomorphisms then the closure of the set of Misiurewicz parameters is equal to $\bif(M).$ 
\end{theorem}
We refer to the original paper \cite{bbd-bif} for more details about the others characterizations of $\bif(M),$ but unlike the one variable case this set is not related to the continuity of the small Julia set in general, see \cite{bianchi-t-desboves}. The same article gives the first example of a family where the bifurcation locus has non-empty interior.

From Theorem \ref{th-misiu}, a natural idea in order to have robust bifurcations in $\Hc_d(\Pb^k)$ is to replace the repelling cycle in the definition of Misiurewicz parameter by a hyperbolic repeller which could be chosen to be large. Let $(f_\lambda)_{\lambda\in M}$ be a holomorphic family of endomorphisms such that $f_{\lambda_0}$ admits a basic repeller $\Lambda_{\lambda_0}.$ We consider the map $F\colon M\times\Pb^k\to M\times\Pb^k$ defined by $F(\lambda,x)=(\lambda,f_\lambda(x))$ and we also denote by $C(F)$ its critical set. Using the notation of Theorem \ref{th-motion} and following \cite{dujardin-bif} we say that $P(f_{\lambda_0})$ \textit{intersects properly} $\Lambda_{\lambda_0}$ if there exist $n\geq1,$ a irreducible component $V_{\lambda_0}$ of $C(f_{\lambda_0})$ and a point $x\in\Lambda_{\lambda_0}$ such that $x\in f^{n}_{\lambda_0}(V_{\lambda_0})$ and arbitrarily close to $\lambda_0$ there is $\lambda\in M$ such that $(\lambda,h_\lambda(x))$ does not belong to the irreducible component of $F^n(C(F))$ containing $\{\lambda_0\}\times f_{\lambda_0}^n(V_{\lambda_0}).$ The next result is a direct consequence of Lemma 2.3 and Proposition-Definition 2.5 of \cite{dujardin-bif}.
\begin{proposition}\label{prop-ouvert-de-bif}
Let $f_0\in\Hc_d(\Pb^k).$ Assume that $f_0$ admits a basic repeller $\Lambda_{f_0}$ contained in $\Jc_k(f_0)$ which has a holomorphic continuation $\Lambda_f$ for $f$ in a neighborhood $\Uc\subset\Hc_d(\Pb^k)$ of $f_0.$ If $P(f)$ intersects $\Lambda_f$ properly for all $f\in\Uc$ then $f_0$ belongs to the interior of the bifurcation locus $\bif(\Hc_d(\Pb^k)).$
\end{proposition}
\begin{proof}
By \cite[Lemma 2.3]{dujardin-bif} (see also \cite[Lemma 2.2.15]{bianchi-phd} or \cite[Exercise 1.60]{ds-lec}) there exists a neighborhood $\Uc'\subset\Uc$ of $f_0$ such that $\Lambda_f\subset\Jc_k(f)$ for all $f\in\Uc'.$ As $P(f)$ intersects properly $\Lambda_f$ for all $f\in\Uc'$ it follows from \cite[Proposition-Definition 2.5]{dujardin-bif} that $\Uc'\subset\bif(\Hc_d(\Pb^k)).$
\end{proof}
\section{Blenders}\label{sec-blender}
This section is devoted to the construction of blenders near some product maps. As a preliminary step, we define particular subsets of $\Cb$ which will be the building blocks for the blenders. The precise definitions of blenders of repelling type and of saddle type will be given in the two corresponding subsections. However, let us emphasize that these definitions are adjusted to our purpose and not as general as possible.

As we have seen, a $f$-invariant set $\Lambda$ is hyperbolic if the tangent bundle over its natural extension splits into two bundles $E^s$ and $E^u$ which are uniformly contracted or expanded by $Df.$ One says that $f$ has $s:=\dim E^s$ \textit{stable directions} and $u:=\dim E^u$ \textit{unstable directions} on $\Lambda.$ Originally, the notion of blenders was introduced for diffeomorphisms on smooth manifolds of dimension larger or equal to $3$ since the construction needs at least $3$ different directions. In our non-invertible setting, it can be started at $k=2$ since the non-injectivity can be consider as an additional stable direction which is especially strong: the preimages of a point $x$ converge in finite time to $x.$ For simplicity, in what follows we only consider the case where $k=2.$ Hence, we will obtain two types of blenders. If $s=0,$ $u=2$ the blender will be of repelling type and if $s=u=1$ then it will be of saddle type. 

All the maps that we will use are perturbations of product maps of the form
$$(z,w)\mapsto(p(z),q(w)).$$
Hence, there are two natural directions. The \textit{horizontal direction} is the one parallel to $\{w=c\}$ and the \textit{vertical direction} is the one parallel to $\{z=c\}.$ The vertical direction will always be close to our strong unstable direction.

Roughly speaking, the idea behind blenders of repelling type for a skew product $f(z,w)=(p(z,w),q(w))$ of $\Cb^2$ is the following. Let $H_1,\ldots, H_N$ and $V_1,\ldots, V_N$ be $2N$ open sets in $\Cb$ and define $H:=\cup_{j=1}^NH_j,$ $V:=\cup_{j=1}^NV_j$ and $Z:=\cup_{j=1}^NH_j\times V_j.$ The set $Z$ contains a blender of repelling type if for each $1\leq j\leq N$
\begin{itemize}
\item $q$ is (strongly) expanding on $V_j$ and $\overline V\subset q(V_j),$
\item $p$ is (weakly) expanding in the horizontal direction on $H_j\times V_j$ and $\overline{p(H_j\times V_j)}\subset H.$
\end{itemize}
Even if $f$ is repelling on $Z$, its geometric behavior and its action on the tangent space (one direction is much more expanded that the other) both mimic a saddle set. And actually, the ``local stable set'' (given by $\Lambda:=\cap_{n\geq0}f^{-n}(\overline Z)$ and which we refer to as the blender) of the maximal invariant set of $f$ in $Z$ behaves as a one dimensional stable manifold: a vertical graph passing through $Z$ has to intersect it (see Proposition \ref{prop-inter-robuste} for a precise statement). Moreover, these properties are stable under small perturbations. These are the main two points about a blender of repelling type: intersection with an open family of graphs and robustness.

A blender of saddle type is more or less a blender of repelling type for $``f^{-1}"$ by taking into account that the non-injectivity can be seen as a strong stable direction.

In what follows, we will always take $N=3$ and for technical reasons (which will be clear in Section \ref{sec-cons}) we may have to exchange $q$ with a large iterate.

For the rest of this section $q$ is an element of $\Pc_d,$ $r_1,$ $r_2,$ $r_3$ are three distinct (but possibly in the same cycle) repelling periodic points of $q$ of period $m_1$ and $\chi_1,$ $\chi_2,$ $\chi_3$ are respectively their multiplier i.e. $\chi_i=(q^{m_1})'(r_i).$ We also assume that they do not belong to the postcritical set of $q.$ The construction of the following open subsets is elementary but will be important for the sequel. Notice that the notations given in this lemma will be used through all this section.
\begin{lemma}\label{le-vois-rep}
There exist $A>0$ and $l_0\geq1$ such that for each $i\in\{1,2,3\}$ there is a sequence $(V_i^l)_{l\geq0}$ of connected neighborhoods of $r_i$ such that $q^{m_1}(V^{l+1}_i)= V^l_i,$ the diameter of $V_i^l$ decreases exponentially fast with $l$ and if $l\geq l_0$ then
\begin{itemize}
\item $|(q^{lm_1})'(w)|\geq A\chi_i^l>1$ for all $w\in V_i^l,$
\item $\overline{\cup_{j=1}^3V^l_j}\subset q^{lm_1}(V_i^l),$
\item $q^{lm_1}$ is a biholomorphism between $W^l_{ij}:=V_i^l\cap q^{-lm_1}(V^l_j)$ and $V^l_j.$
\end{itemize}
\end{lemma}
\begin{proof}
For $i\in\{1,2,3\},$ let $U_i$ be a connected neighborhood of $r_i$ on which $q^{m_1}$ is conjugated to $w\mapsto\chi_iw$ and let $\phi_i$ be its local inverse on $U_i.$ There exists $l_1\geq1$ such that $\cup_{j=1}^3U_j\subset q^{l_1m_1}(U_i).$ By shrinking $U_i$ if necessary, we can assume that $q^{-l_1m_1}(r_j)\cap \overline{U_i}$ is reduced to one point. Since the points $r_j$ don't belong to the postcritical set of $q,$ there also exists $l_2\geq1$ such that $\phi_i^{l_2}(U_i)$ doesn't contain critical values of $q^{l_1m_1}.$ Hence, there exists a small neighborhood $N$ of the critical set of $q^{l_1m_1}$ such that $V_i:=U_i\setminus N$ is connected and satisfies $\cup_{j=1}^3\phi_j^{l_2}(U_j)\subset q^{l_1m_1}(V_i).$ Thus, for $l\geq0$ we can define $V_i^{l_1+l}:=\phi_i^l(V_i)$ and $V_i^{l_1-l}:=q^{lm_1}(V_i).$ Since $q^{m_1}$ is conjugated to $w\mapsto\chi_iw$ on $U_i$ it is easy to check that the proposition holds if $l_0\geq l_1+l_2$ is large enough.
\end{proof}
From this we can obtain hyperbolic basic sets $\Lambda^l:=\cap_{i\geq0}q^{-ilm_1}(V_1^l\cup V_2^l\cup V_3^l).$ The third point in the lemma insures that $q^{lm_1}_{|\Lambda^l}$ is conjugated to a shift of three symbols. The repelling blenders constructed in what follows will project injectively into one of these sets $\Lambda^l.$ Thus they are contained in a Cantor set and their dynamics are conjugated to subshifts.

In the horizontal direction, for both the repelling and the saddle case, we will eventually use the following subsets which depend on the configuration of three complex numbers $c_1,c_2,c_3\in\Cb.$ Here, $\Db$ is the unit disc of $\Cb$ and if $E\subset\Cb$ and $t\in\Cb^*$ then $tE$ denotes the image of $E$ by a homothety of ratio $t.$
\begin{lemma}\label{le-triangle}
Let $c_1,c_2,c_3\in\Cb$ be three non-aligned points such that $c_1+c_2+c_3=0.$ There exist $\epsilon_0>0,$ $\alpha_0>0$ arbitrarily small such that if $\phi_j(z):=\rho z+\epsilon_0c_j$ with $1-\alpha_0\leq|\rho|\leq1+\alpha_0,$ $j\in\{1,2,3\},$ then
$$\overline{\Db}\subset\cup_{j=1}^3\phi_j(\Db),$$
and there are three open sets $H_1,$ $H_2$ and $H_3$ such that $\Db=\cup_{j=1}^3H_j,$ $\overline{\phi_j(H_j)}\subset\Db$ and $\overline{\phi_j(1/3H_j)}\subset1/3\Db.$
\end{lemma}
\begin{proof}
Since $c_1,$ $c_2$ and $c_2$ define a non degenerated triangle with $c_1+c_2+c_3=0$ it follows that the three half-planes $\{z\in\Cb\,|\, \mathrm{Re}(zc_j^{-1})>0\}$ cover $\Cb^*.$ Therefore, for $\eta>0$ small enough and $\rho\in\Cb^*$ the sets
$$H_j=(2^{-1}\Db)\cup\{z\in\Db\,|\, |\arg(-z\rho c_j^{-1})|<\pi/2-\eta\}$$
satisfy $\Db=\cup_{j=1}^3H_j.$ Here, $\arg(z)$ denotes the argument of $z$ with value in $]-\pi,\pi].$ On the other hand, if $\epsilon_0>0$ and $\alpha_0>0$ are small enough then $\overline{\phi_j(H_j)}\subset\Db.$ Actually, to see this for $j=1$ we conjugate by a rotation in order to have $c_1\in\Rb_{>0}.$ Then $\overline{\phi_1(H_1)}\subset\Db$ is equivalent to $\phi_1(-e^{i(\pi/2-\eta)-\arg(\rho)})\in\Db$ which holds if $(1+\alpha_0)^2+(\epsilon_0c_1)^2-2\epsilon_0c_1(1+\alpha_0)\cos(\pi/2-\eta)<1.$ This inequality is satisfied if first we choose $\epsilon_0>0$ small and then $\alpha_0>0$ close to $0.$ The last inclusion follows with exactly the same arguments.

To conclude, observe that we can apply the proof above to $\phi_j^{-1}$ (with the adequate $H_j$) which gives $\overline\Db\subset\cup_{j=1}^3\phi_j(\Db).$
\end{proof}
Observe that $H_j$ depends on the argument of $\rho$ but not $\phi_j(H_j).$ Hence, for simplicity we will not notice this dependency.

\subsection{Repelling type}
For $\delta>0,$ define the cone $C_\delta:=\{(a,b)\in\Cb^2\, |\, |a|<\delta|b|\}.$ The natural identification between the tangent space of $x\in\Cb^2$ and $\Cb^2$ allows us to consider $C_\delta$ as a (constant) cone field over $\Cb^2.$ We say that a smooth self-map $g$ of $\Cb^2$ \textit{contracts the cone field} $C_\delta$ on $U\subset\Cb^2$ if there exists $0<\delta'<\delta$ such that $D_xg(C_\delta)\subset C_{\delta'}$ for all $x\in U.$

In order to state the next result we choose three arbitrary open sets $H_1,H_2,H_3\subset\Cb$ and we set $H:=\cup_{i=1}^3H_i$ and $Z^l:=\cup_{i=1}^3
(H_i\times V_i^l).$
\begin{proposition}\label{prop-inter-robuste}
Let $K$ be a compact subset of $H.$ There exist $\delta_0>0$ and $l_0\geq1$ such that if $l\geq l_0$ and $g\colon Z^l\to K\times\Cb$ is a holomorphic map of the form $g(z,w)=(h(z,w),q^{lm_1}(w))$ which contracts $C_{\delta_0}$ then there is $\alpha>0$ with the following property. If $f$ is another holomorphic map defined on $Z^l$ such that $\|f-g\|_{C^1,Z^l}\leq\alpha$ and if $\sigma\colon V^l_i\to H_i$ is holomorphic with $\|\sigma'\|_{C^0,V^l_i}<\delta_0$ then $\Lambda_f:=\cap_{n\geq0}f^{-n}(\overline{Z^l})$ intersects the vertical graph $\Gamma_\sigma:=\{(\sigma(w),w)\,|\, w\in V^l_i\}$ of $\sigma.$
\end{proposition}
The same statement holds for smooth maps $f,$ $g,$ and $\sigma$ but we only state it for holomorphic maps for convenience. Observe that the important point in this result for what follows is that $\delta_0$ is independent of $l\geq l_0.$
\begin{proof}
Let $K$ be a compact subset of $H$ and let $l_0\geq1$ be as in Lemma \ref{le-vois-rep}. We will show that if $\delta_0>0$ is small enough and $l\geq l_0$ then the proposition holds. Let $l\geq l_0$ and $g(z,w)=(h(z,w),q^{lm_1}(w))$ be as above. For simplicity, we first consider the unperturbed case. Let $i_0\in\{1,2,3\}$ and let $\sigma\colon V^l_{i_0}\to H_{i_0}$ be a holomorphic map with $\|\sigma'\|_{C^0,V^l_{i_0}}<\delta_0.$ Since $K$ is a compact subset of $\cup_{i=1}^3H_i,$ there exists $\eta>0$ such that if $z\in K$ then there is $i\in\{1,2,3\}$ with $D(z,2\eta)\subset H_i.$ In particular, such ${i_1}$ exists for $h(\sigma(r_{i_0}),r_{i_0})$ since $h(Z^l)\subset K.$ We denote by $\psi_{ij}\colon V_{j}^l\to W^l_{ij}\subset V^l_{i}$ the inverse branch of $q^{lm_1}$ given by Lemma \ref{le-vois-rep}. Define $\sigma_1$ on $V^l_{i_1}$ by $\sigma_1(w):=h(\sigma\circ\psi_{{i_0}{i_1}}(w),\psi_{{i_0}{i_1}}(w)).$ By definition, the graph $\Gamma_{\sigma_1}$ is contained in the image by $g$ of the graph $\Gamma_\sigma.$ Since the vectors tangent to $\Gamma_\sigma$ belongs to $C_{\delta_0}$ and $g$ contracts $C_{\delta_0}$ the same holds for vectors tangent to $g(\Gamma_\sigma)$ which contains $\Gamma_{\sigma_1}$ i.e. $\|\sigma_1'\|_{C^0,V^l_{i_1}}<\delta_0.$ We claim that if $\delta_0$ is small enough (uniformly on $l\geq l_0$) then $\sigma_1(V_{i_1}^l)\subset H_{i_1}.$ This simply comes from the fact that $g(\Gamma_\sigma)$ is tangent to $C_{\delta_0}$ but since this point is important we give the details.

By the proof of Lemma \ref{le-vois-rep}, $q^{(l-l_0)m_1}$ is a biholomorphism between $V^l_{i_0}$ and $V^{l_0}_{i_0}$ and we denote by $\phi$ its inverse. On the other hand, there exists a constant $C>0$ such that if $w\in V_{i_0}^{l_0}$ then there is a smooth path $\gamma$ between $r_{i_0}$ and $w$ of length smaller than $C.$ The length of $\widetilde\gamma:=q^{l_0m_1}\circ\gamma$ is also bounded by a uniform constant $\widetilde C>0.$ If $\widehat\gamma(t):=h(\sigma\circ\phi\circ\gamma(t),\phi\circ\gamma(t))$ then $t\mapsto(\widehat\gamma(t),\widetilde\gamma(t))$ gives a path on $g(\Gamma_\sigma).$ Hence, the length of $\widehat\gamma$ is bounded by $\delta_0$ times the one of $\widetilde\gamma$ and thus by $\delta_0\widetilde C.$ Since $w\in V^{l_0}_{i_0}$ was arbitrary it follows that the image of $V^l_{i_0}$ by $w\mapsto h(\sigma(w),w)$ is contained in $D(h(\sigma(r_{i_0}),r_{i_0}),\delta_0\widetilde C).$ Therefore, if $\delta_0<\eta\widetilde C^{-1}$ then $D(h(\sigma(r_{i_0}),r_{i_0}),\delta_0\widetilde C)\subset D(h(\sigma(r_{i_0}),r_{i_0}),\eta)\subset H_{i_1}$ which gives the claim.

We can now conclude the proof in the unperturbed case. As $g(\Gamma_\sigma)$ contains the graph of a map $\sigma_1$ with the same properties than $\sigma,$ we can define inductively a sequence $(i_n)_{n\geq1}$ of $\{1,2,3\}$ and maps $\sigma_n\colon V_{i_n}^l\to H_{i_n}$ such that $\Gamma_{\sigma_n}\subset g(\Gamma_{\sigma_{n-1}}).$ By construction, $\Gamma_{\sigma_n}$ is contained in $Z^l$ and is the image by $g^n$ of the graph over $V^l_{i_n}$ of $\sigma\circ\psi_{i_0i_1}\circ\cdots\circ\psi_{i_{n-1}i_n}.$ The sequence of subsets of $V^l_{i_0}$ given by $\psi_{i_0i_1}\circ\cdots\circ\psi_{i_{n-1}i_n}(V^l_{i_n})$ is decreasing and converges to a point $w\in V^l_{i_0}.$ Thus, the point $x:=(\sigma(w),w)\in\Gamma_\sigma$ satisfies $g^n(x)\in\Gamma_{\sigma_n}\subset Z^l$ for all $n\geq0,$ i.e. $x\in\Lambda_g\cap\Gamma_\sigma.$

The perturbed case then follows easily since the size $\alpha$ of the perturbations can depend on $g$ and thus on $l.$ If $f$ is sufficiently $C^1$-close to $g$ then the image of $\Gamma_\sigma$ by $f$ is a small deformation of $g(\Gamma_\sigma)$ and thus contains a graph of a map $\tau_1$ defined on $V^l_{i_1}$ which is close to $\sigma_1.$ Since $\sigma_1(V^l_{i_1})\subset D(h(\sigma(r_{i_0}),r_{i_0}),\eta)\subset D(h(\sigma(r_{i_0}),r_{i_0}),2\eta)\subset H_{i_1}$ then $\tau_1(V^l_{i_1})\subset H_{i_1}$ if $f$ is sufficiently $C^1$-close to $g.$ Then, following this procedure we construct a sequence of maps $\tau_n$ such that $\Gamma_{\tau_n}\subset f(\Gamma_{\tau_{n-1}})\cap Z^l$ and conclude the proof as above.
\end{proof}

In other words, Proposition \ref{prop-inter-robuste} says that the set $\Lambda_f$ intersects any vertical graph tangent to $C_{\delta_0}$ in $H_i\times V^l_i.$ In particular, with $\sigma$ constant we obtain that $H$ is contained in the projection of $\Lambda_f$ on the horizontal direction. However, Proposition \ref{prop-inter-robuste} gives a much more robust property which will be the key ingredient to prove Corollary \ref{coro-bif}. Actually, this property can be seen as the main part of the definition of blenders of repelling type.
\begin{definition}\label{def-blen-rep}
If $f$ and $\Lambda_f$ are as in Proposition \ref{prop-inter-robuste} and if $\Lambda_f$ is contained in a repelling hyperbolic set of $f$ then we say that $\Lambda_f$ is a blender of repelling type.
\end{definition}
In the sequel, we will use such blenders to ``blend'' robustly the postcritical set $P(f)$ of a map $f$ with its small Julia set $\Jc_k(f).$ We will obtain them as infinite intersections $\Lambda_f=\bigcap_{n\geq0}f^{-n}(\overline{Z^l}),$ for some set $Z^l.$ Thus, $f(\Lambda_f)\subset\Lambda_f$ but without equality in general. Moreover, such intersections of compacts do not behave well under perturbations. It is for these reasons that we require $\Lambda_f$ to be contained in a hyperbolic repeller since these sets are invariant and can be followed holomorphically under perturbations. 

In what follows, we obtain blenders of repelling type using the subsets defined in Lemma \ref{le-vois-rep} and Lemma \ref{le-triangle}. As above, let $c_1,$ $c_2$ and $c_3$ be three non-aligned points in $\Cb$ such that $c_1+c_2+c_3=0.$ It follows from Proposition \ref{prop-inter-robuste} that good perturbations of the model map $(z,w)\mapsto(\rho z+\epsilon_0c_j,q^{m_1l}(w))$ on $2\Db\times V_j^l$ have a blender.

\begin{proposition}\label{prop-exi-rep}
Let $\epsilon_0>0,$ $\alpha_0>0$ be as in Lemma \ref{le-triangle}. There exist $l_0\geq1$ and $\delta_0>0$ such that if $l\geq l_0$ and $\rho\in\Cb$ satisfies $|\rho|=1+\alpha_0$  then a polynomial skew product $g$ of the form $g(z,w)=(h(z,w),q^{m_1l}(w))$ with
\begin{itemize}
\item[i)] $|h(z,w)-(\rho z+\epsilon_0 c_j)|\leq\delta_0$ for all $(z,w)\in2\Db\times V_j^{l}$ and $j\in\{1,2,3\},$
\item[ii)] $g$ contracts the cone field $C_{\delta_0}$ on $\overline{\cup_{j=1}^3\Db\times V^{l}_j},$
\end{itemize}
has two nested blenders of repelling type $\Lambda'_g\subset\Lambda_g.$ More precisely, there exist subsets $H_1,$ $H_2$ and $H_3$ of $\Db$ such that $1/2\Db\subset\cap_{j=1}^3H_j$ and
$$\Lambda'_g=\bigcap_{i\geq0}g^{-i}\left(\overline{\bigcup_{j=1}^3\left(\frac{1}{3}H_{j}\right)\times V^l_j}\right)\ \text{ and }\ \Lambda_g=\bigcap_{i\geq0}g^{-i}\left(\overline{\bigcup_{j=1}^3H_{j}\times V^l_j}\right).$$
\end{proposition}
The fact that we need two nested blenders is technical. The smallest one $\Lambda'_g$ could be replaced by any uncountable subset of $\Lambda_g$ whose holomorphic continuation for $f$ close enough to $g$ belongs to $\bigcap_{i\geq0}f^{-i}\left(\overline{\cup_{j=1}^3H_{j}\times V^l_j}\right).$
\begin{proof}
Let $\rho\in\Cb$ be such that $|\rho|=1+\alpha_0.$ We consider the maps $\phi_j(z):=\rho z+\epsilon_0c_j$ and the three open sets $H_j,$ $j\in\{1,2,3\},$ as in Lemma \ref{le-triangle}. Since $\overline{\phi_j(H_j)}\subset\Db$ for each $j=1,2,3,$ there exists $\eta>0$ such that $\phi_j(H_j)\subset D(0,1-2\eta).$ Let $\delta_0>0$ and $l_0\geq1$ be the constants obtained by Proposition \ref{prop-inter-robuste} with $K=\overline{D(0,1-\eta)}.$ We can assume $\delta_0<\eta.$ It follows that if $l\geq l_0$ and $g$ satisfies the assumptions of the present proposition, it also satisfies those of Proposition \ref{prop-inter-robuste} on $\cup_{j=1}^3H_j\times V_j^l.$ By taking a smaller $\delta_0$ if necessary, Proposition \ref{prop-inter-robuste} also applies to $\cup_{j=1}^3(1/3H_j)\times V_j^l.$

On the other hand, $g$ is repelling on $\overline{\cup_{j=1}^3\Db\times V^l_j}.$ Indeed, its differential is equal to
$$\begin{pmatrix}
\partial_zh(z,w)&\partial_wh(z,w)\\
0&(q^{lm_1})'(w)
\end{pmatrix}.$$
By Lemma \ref{le-vois-rep}, if $l\geq l_0$ then $|(q^{lm_1})'(w)|>1$ on each $V_j^l.$ And since $|h(z,w)-\phi_j(z)|\leq\delta_0$ on $2\Db\times V_j^{l},$ Cauchy's inequality gives $|\partial_zh(z,w)-\rho|\leq\delta_0$ on $\overline{\Db\times V_j^{l}}.$ Thus, $|\partial_zh(z,w)|>1$ if $\delta_0$ is small enough. Furthermore, by Lemma \ref{le-triangle}, $\overline\Db\subset\cup_{j=1}^3\phi_j(\Db)$ and thus there exists $r>0$ such that for each $z_0\in\overline{\Db}$ there is $j_0\in\{1,2,3\}$ with $D(z_0,r)\subset\phi_{j_0}(\Db).$ If $w_0\in\cup_{j=1}^3V_j^l$ then Lemma \ref{le-vois-rep} gives $w_1\in V_{j_0}^l$ with $q^l(w_1)=w_0.$ Hence, if $\delta_0<r$ then by Rouché's theorem there is $z_1\in\Db$ such that $h(z_1,w_1)=z_0,$ i.e. $g(z_1,w_1)=(z_0,w_0)$ and thus $\overline{\cup_{j=1}^3\Db\times V_j^l}\subset g\left({\cup_{j=1}^3\Db\times V_j^l}\right).$ Hence, $\Lambda_g$ is contained in the invariant repelling set $\bigcap_{i\geq0}g^{-i}\left(\overline{\cup_{j=1}^3\Db\times V_j^l}\right).$
\end{proof}

\begin{remark}
An important point is that no estimates on $\partial_wh(z,w)$ are required i.e. $h$ can be far from $\phi_j$ in the $C^1$ topology as long as $g$ contracts the cone field $C_{\delta_0}.$ A priori, it is the situation that occurs in Section \ref{sec-cons}.
\end{remark}

\subsection{Saddle type}
On $\Pb^2$ the blenders of saddle type are simpler than the repelling ones. Indeed, no cone field condition is required since one stable direction is given by preimages on whose the contraction dominates any ``standard'' stable direction.
\begin{definition}\label{def-blen-selle}
Let $f$ be in $\Hc_d(\Pb^2)$ and let $Z$ be an open subset of $\Pb^2$ such that $\overline{Z}\subset f(Z).$ If for all $g\in\Hc_d(\Pb^2)$ sufficiently close to $f,$ $\Lambda_g:=\bigcap_{n\geq0}g^{-n}(\overline Z)$ is a saddle hyperbolic set then we say that $\Lambda_g$ is a blender of saddle type.
\end{definition}
This definition is artificially stable under small perturbations. The examples obtained with Proposition \ref{prop-exi-selle} will automatically satisfy this stability condition since they preserve a dominated splitting on a neighborhood of $\overline Z.$ Another way to fulfill this condition is that $\Lambda_f$ is contained in a locally maximal saddle hyperbolic set $\widetilde\Lambda_f.$ In this case, standard results (see \cite[Corollary 2.6]{jonsson-phd}) imply that $Z$ is contained in the unstable set of $\widetilde\Lambda_f.$ This remark will be use in the proof of Theorem \ref{th-cycle}.

The following result is the counterpart of Proposition \ref{prop-exi-rep} in the saddle case and its proof is identical.
\begin{proposition}\label{prop-exi-selle}
Let $\epsilon_0>0,$ $\alpha_0>0$ be as in Lemma \ref{le-triangle}. There exist $l_0\geq1$ and $\delta_0>0$ such that if $l\geq l_0$ and $\rho\in\Cb$ satisfies $|\rho|=1-\alpha_0$  then a polynomial skew product $g$ of the form $g(z,w)=(h(z,w),q^{m_1l}(w))$ with $|h(z,w)-(\rho z+\epsilon_0 c_j)|\leq\delta_0$ for all $(z,w)\in2\Db\times V_j^{l}$ and $j\in\{1,2,3\},$ has a blender of saddle type $\Lambda_g:=\cap_{n\geq0}g^{-n}(\overline{\cup_{i=1}^3\Db\times V^l_i}).$
\end{proposition}

\section{Existence near product bifurcations}\label{sec-cons}
It follows almost immediately from Proposition \ref{prop-exi-rep} and Proposition \ref{prop-exi-selle} that if $q$ has three non-aligned repelling periodic points which are not postcritical then for $l$ large some perturbations of $(z,w)\mapsto(z,q^l(w))$ possess blenders. The aim of this section is to generalize this to $(p(z),q(w))$ when $p$ has a parabolic cycle and $q$ is arbitrary.

In order to fix the setting and the notations of this section let $(p_\lambda)_{\lambda\in M}$ be a holomorphic family of degree $d$ polynomials parametrized by a complex manifold $M$ and such that $0$ is a periodic point of $p_\lambda$ for all $\lambda\in M.$ We denote by $m_0$ the period of $0$ and by $\rho(\lambda)$ its multiplier, i.e. $\rho(\lambda):=(p_\lambda^{m_0})'(0).$ We assume that there exists $\lambda_0\in M$ such that $\rho(\lambda_0)$ is a primitive $t_0$-th root of unity and we define, for $-1\leq i\leq m_1-1$ with $m_1:=m_0t_0,$
\begin{equation}\label{eq-bi}
b_i(\lambda):=(p_\lambda^{m_1-i-1})'(p_\lambda^{i+1}(0)).
\end{equation}
Observe that $b_{-1}(\lambda)=\rho(\lambda)^{t_0}$ and $b_i(\lambda)=\rho(\lambda)b_{i+m_0}(\lambda).$ We use this last formula to extend the definition of $b_i(\lambda)$ for $i\in\Zb$ if $\lambda$ is close to $\lambda_0.$ In particular,  $b_{-1}(\lambda_0)=1$ and $b_i(\lambda_0)\neq0$ for all $i\in\Zb.$ Finally, let $q$ be an element of $\Pc_d$ and let $r_1,$ $r_2$ and $r_3$ be three repelling points of period $m_1$ of $q$ which are not in the postcritical set of $q.$ We set
\begin{equation}\label{eq-ci}
c_j:=\sum_{i=0}^{m_1-1}b_i(\lambda_0)q^i(r_j)
\end{equation}
for $j\in\{1,2,3\}.$ We deduce from Proposition \ref{prop-exi-rep} and Proposition \ref{prop-exi-selle} the following result.
\begin{theorem}\label{th-constru}
Assume that the three complex numbers $c_1,$ $c_2$ and $c_3$ are not aligned and that $c_1+c_2+c_3=0.$ If $(\lambda_n)_{n\geq1}$ is a sequence in $M$ converging to $\lambda_0$ such that $|\rho(\lambda_n)|>1$ (resp. $|\rho(\lambda_n)|<1)$ then there exist a sequence $(l_n)_{n\geq1}$ of integers and a sequence $(\beta_n)_{n\geq1}$ of $\Cb^*$  such that $\lim_{n\to\infty}\beta_n=0$ and for all $\alpha\in\Db^*$ the $(m_1l_n)$-th iterate of
$$f_{n,\alpha}(z,w):=(p_{\lambda_n}(z)+\alpha\beta_nw,q(w))$$
has a blender $\Lambda_{n,\alpha}$ of repelling type (resp. of saddle type) for $n$ large enough.
\end{theorem}

To prove the theorem, we have to understand how a small perturbation of the product map $(p_\lambda(z),q(w))$ behaves under iterations near the line $\{z=0\}.$ To this aim we will change coordinates in order to focus only on the dynamics close to this line.

Let $\lambda\in M,$ $\epsilon\in\Cb^*$ and $s\in\Db^*.$ We consider $\lambda$ and $\epsilon$ as constants and $s$ as a variable. We define the skew product of $\Cb^2$
$$f(z,w):=(p_\lambda(z)+s\epsilon w,q(w))$$
and its renormalisation by $\phi(z,w):=(sz,w),$ $g:=\phi^{-1}\circ f\circ\phi.$ If $p_\lambda$ is of the form $p_\lambda(z)=\sum_{i=0}^da_i(\lambda)z^i$ then $g$ satisfies
$$g(z,w)=(a_0(\lambda)s^{-1}+a_1(\lambda)z+\epsilon w+sE(s,z,w),q(w)),$$
where $E$ is a polynomial. The following proposition gives estimates on the iterates of $g.$ 
\begin{proposition}\label{prop-esti}
If $l\geq1$ then there exists a polynomial $E_l$ in $s,$ $z$ and $w$ such that
$$g^{m_1l}(z,w)=\left(\rho(\lambda)^{t_0l}z+\epsilon\sum_{i=0}^{m_1-1}b_i(\lambda)\sum_{k=0}^{l-1}\rho(\lambda)^{t_0k}q^{i+(l-1-k)m_1}(w)+sE_l(s,z,w),q^{m_1l}(w)\right).$$
\end{proposition}
\begin{proof}
Since $g$ is a skew product, $g^k$ is of the form $g^k(z,w)=(g_k(z,w),q^k(w))$ with
$$g_1(z,w)=\sum_{i=0}^da_i(\lambda)s^{i-1}z^i+\epsilon w \text{ \  and \ } g_{k+1}(z,w)=g_1(g_k(z,w),q^k(w)).$$
Furthermore, there exists functions $g_{k,i}$ independent of $s$ such that
$$g_k(z,w)=\sum_{i=0}^{d^k}g_{k,i}(z,w)s^{i-1}.$$
We are only interested in the sequences $(g_{k,0})_{k\geq0}$ and $(g_{k,1})_{k\geq0}$ that can be computed inductively. We easily have that $g_{0,0}(z,w)=0$ and $g_{0,1}(z,w)=z.$ Using that $g_{k+1}(z,w)=g_1(g_k(z,w),q^k(w))$ we obtain that
$$g_{k+1,0}(z,w)=\sum_{i=0}^da_i(\lambda)g_{k,0}(z,w)^i$$
and
$$g_{k+1,1}(z,w)=g_{k,1}(z,w)\sum_{i=0}^dia_i(\lambda)(g_{k,0}(z,w))^{i-1}+\epsilon q^k(w),$$
i.e.  $g_{k+1,0}(z,w)=p_\lambda(g_{k,0}(z,w))$ and $g_{k+1,1}(z,w)=g_{k,1}(z,w)p_\lambda'(g_{k,0}(z,w))+\epsilon q^k(w).$ It follows that $g_{k,0}(z,w)=p_\lambda^k(0)$ and
$$g_{k,1}(z,w)=z\prod_{i=0}^{k-1}p_\lambda'(p_\lambda^i(0))+\epsilon\sum_{i=0}^{k-1}q^i(w)\prod_{j=i+1}^{k-1}p_\lambda'(p_\lambda^j(0)).$$
In particular, using that $0$ is $m_0$-periodic for $p_\lambda$ and the definition of $b_i(\lambda)$ given in \eqref{eq-bi}, taking $k=m_1:=m_0t_0$ gives $g_{m_1,0}(z,w)=0$ and $g_{m_1,1}(z,w)=\rho(\lambda)^{t_0}z+\epsilon\sum_{i=0}^{m_1-1}b_i(\lambda)q^i(w).$ Therefore, we have
$$g^{m_1}(z,w)=\left(\rho(\lambda)^{t_0}z+\epsilon\sum_{i=0}^{m_1-1}b_i(\lambda)q^i(w)+sE_1(s,z,w),q^{m_1}(w)\right)$$
for some polynomial $E_1.$ The proposition follows easily by iterating this formula.
\end{proof}

\begin{proof}[Proof of Theorem \ref{th-constru}]
Assume as in the statement that $c_1,c_2$ and $c_3$ defined in \eqref{eq-ci} are not aligned with $c_1+c_2+c_3=0$ and that $(\lambda_n)_{n\geq1}$ is a sequence which converges to $\lambda_0$ with $|\rho(\lambda_n)|>1$ (resp. $|\rho(\lambda_n)|<1).$

First, we have to fix all the constants which will be involved in the perturbations. Let $l_0\geq1,$ $\epsilon_0>0,$ $\delta_0>0$ and $\alpha_0>0$ be the constants given by Proposition \ref{prop-exi-rep} (resp. Proposition \ref{prop-exi-selle}). Since $\rho(\lambda_n)^{t_0}$ converges to $1$, there exists a sequence of integers $(l_n)_{n\geq1}$ which are all larger than $l_0$ and such that $|\rho(\lambda_n)^{t_0l_n}|$ converges to $1+\alpha_0$ (resp. $1-\alpha_0$). Therefore, there exist $\rho_n\in\Cb$ with $|\rho_n|=1+\alpha_0$ (resp. $1-\alpha_0$) such that $|\rho(\lambda_n)^{t_0l_n}-\rho_n|$ converges to $0.$ Finally, the sequence defined by
$$\epsilon_n:=\epsilon_0\frac{\rho(\lambda_n)^{t_0}-1}{\rho(\lambda_n)^{t_0l_n}-1}$$
converges to $0$.

Now, we consider the neighborhoods $V_j^{l_n}$ of $r_j,$ $j\in\{1,2,3\},$ obtained by Lemma \ref{le-vois-rep}. As their diameters decrease exponentially fast, there exist $a_1>0$ and $0<a_2<1$ such that ${\rm Diam}(V_j^{l_n})\leq a_1a_2^{l_n}$ for all $n\geq1$ and $j\in\{1,2,3\}.$ By Proposition \ref{prop-esti}, if $s\in\Db^*$ then the renormalization $g_n:=\phi^{-1}\circ f_n\circ\phi$ by $\phi(z,w)=(sz,w)$ of
$$f_n(z,w):=(p_{\lambda_n}(z)+s\epsilon_nw,q(w))$$
satisfies
\begin{align*}
&g_n^{m_1l_n}(z,w)=\\
&\left(\rho(\lambda_n)^{t_0l_n}z+\epsilon_n\sum_{i=0}^{m_1-1}b_i(\lambda_n)\sum_{k=0}^{l_n-1}\rho(\lambda_n)^{t_0k}q^{i+(l_n-1-k)m_1}(w)+sE_{l_n}(s,z,w),q^{m_1l_n}(w)\right).
\end{align*}
We first show that if $n\geq1$ is large enough then $g_n^{m_1l_n}$ is arbitrarily close to
$$(\rho_n z+\epsilon_0 c_j,q^{m_1l_n}(w))$$
on $2\Db\times V^{l_n}_j.$ To this aim, observe that $c_j(\lambda_n):=\sum_{i=0}^{m_1-1}b_i(\lambda_n)q^i(r_j)$ converges to $c_j=\sum_{i=0}^{m_1-1}b_i(\lambda_0)q^i(r_j).$ And, by definition, we have
$$\epsilon_0c_j(\lambda_n)=\epsilon_n\frac{\rho(\lambda_n)^{t_0l_n}-1}{\rho(\lambda_n)^{t_0}-1}c_j(\lambda_n)=\epsilon_n\sum_{i=0}^{m_1-1}b_i(\lambda_n)\sum_{k=0}^{l_n-1}\rho(\lambda_n)^{t_0k}q^i(r_j).$$
On the other hand, since $q^{km_1}(V_j^l)=V^{l-k}_j$, we deduce from the estimates on the diameters of $V^l_j$ that  there exists a constant $C>1$ depending only on $q$ and $m_1$ such that for all $w\in V_j^{l_n}$ and $0\leq i\leq m_1-1$ we have
$$|q^{i+(l_n-1-k)m_1}(w)-q^i(r_j)|\leq Ca_1a_2^{k+1}.$$
Hence, if $w\in V_j^{l_n}$ then for $C>1$ chosen large enough we have
\begin{align*}
&\left|\left(\epsilon_n\sum_{i=0}^{m_1-1}b_i(\lambda_n)\sum_{k=0}^{l_n-1}\rho(\lambda_n)^{t_0k}q^{i+(l_n-1-k)m_1}(w)\right)-\epsilon_0c_j(\lambda_n)\right|\\
&\leq\left|\epsilon_n\sum_{i=0}^{m_1-1}b_i(\lambda_n)\left(\sum_{k=0}^{l_n-1}\rho(\lambda_n)^{t_0k}q^{i+(l_n-1-k)m_1}(w)-q^i(r_j)\right)\right|\\
&\leq|\epsilon_n|\sum_{i=0}^{m_1-1}|b_i(\lambda_n)|\sum_{k=0}^{l_n-1}|\rho(\lambda_n)^{t_0k}||q^{i+(l_n-1-k)m_1}(w)-q^i(r_j)|\\
&\leq C^2a_1a_2|\epsilon_n|\sum_{k=0}^{l_n-1}(a_2|\rho(\lambda_n)^{t_0}|)^k=C^2a_1a_2|\epsilon_n|\frac{(a_2|\rho(\lambda_n)^{t_0}|)^{l_n}-1}{(a_2|\rho(\lambda_n)^{t_0}|)-1}=:C_n.
\end{align*}
The last term $C_n$ converges to $0$ since $0<a_2<1.$ Therefore, on $2\Db\times V_j^{l_n}$ we have 
\begin{align*}
\|g_n^{m_1l_n}(z,w)&-(\rho_n z+\epsilon_0 c_j,q^{m_1l_n}(w))\|\\
&\leq 2|\rho(\lambda_n)^{t_0l_n}-\rho_n|+\epsilon_0|c_j-c_j(\lambda_n)|+C_n+|sE_{l_n}(s,z,w)|.
\end{align*}
Observe that for each $n\geq1$ there exists $0<s_n<1$ such that if $|s|\leq s_n$ then $|sE_{l_n}(s,z,w)|\leq\delta_0/2$ on $2\Db\times V_j^{l_n}.$ On the other hand, for $n\geq1$ large enough we have
$$2|\rho(\lambda_n)^{t_0l_n}-\rho_n|+\epsilon_0|c_j-c_j(\lambda_n)|+C_n\leq\delta_0/2.$$
By Proposition \ref{prop-exi-selle}, this is sufficient to conclude in the saddle case (i.e. $|\rho(\lambda_n)|<1$): by removing the normalization, if $\beta_n:=\epsilon_ns_n$ and $\alpha=s/\beta_n$ then for $n\geq1$ large enough the $(m_1l_n)$-th iterate of
$$f_{n,\alpha}(z,w):=(p_{\lambda_n}(z)+\alpha\beta_nw,q(w))$$
has a blender of saddle type contained in $\overline{\cup_{j=1}^3D(0,|\alpha\beta_n|)\times V_j^{l_n}}.$

For the repelling case ($|\rho(\lambda_n)|>1$), it remains to show that $g^{m_1l_n}_n$ contracts the cone field $C_{\delta_0}$ on $\overline{\cup_{j=1}^3\Db\times V_j^{l_n}}$ if $n\geq1$ is large enough. The formula above for $g_n^{m_1l_n}$ gives
$$D_{(z,w)}g_n^{m_1l_n}=\begin{pmatrix}
\rho(\lambda_n)^{t_0l_n}+s\partial_zE_{l_n}(z,w)& a_n(w)+s\partial_wE_{l_n(z,w)}\\\
0& (q^{m_1l_n})'(w)
\end{pmatrix},$$
where $a_n(w):=\epsilon_n\sum_{i=0}^{m_1-1}b_i(\lambda_n)\sum_{k=0}^{l_n-1}\rho(\lambda_n)^{t_0k}(q^{i+(l_n-1-k)m_1})'(w).$ By Lemma \ref{le-vois-rep}, if $\chi_j$ denotes the multiplier of $r_j$ then there is a constant $A>0$ such that for all $w\in V_j^{l_n}$ 
$$|(q^{m_1l_n})'(w)|\geq A\chi_j^{l_n},\ \ |\rho(\lambda_n)^{t_0l_n}|\leq B\  \text{ and }\ |a_n(w)|\leq B|\epsilon_n|\chi_j^{l_n},$$
where $B>0$ is another constant. Since $\epsilon_n$ converges to $0,$ by taking smaller $0<s_n<1$ if necessary, we obtain that for $n\geq1$ large enough and $|s|\leq s_n$ the map $g_n^{m_1l_n}$ contracts $C_{\delta_0}$ on $\overline{\cup_{j=1}^3\Db\times V_j^{l_n}}.$ Hence, by Proposition \ref{prop-exi-rep}, if $\beta_n:=\epsilon_ns_n$ and $\alpha=s/\beta_n$ then for $n\geq1$ large enough the $(m_1l_n)$-th iterate of
$$f_{n,\alpha}(z,w):=(p_{\lambda_n}(z)+\alpha\beta_nw,q(w))$$
has a blender of repelling type contained in $\overline{\cup_{j=1}^3D(0,|\alpha\beta_n|)\times V_j^{l_n}}.$
\end{proof}
To conclude the proof of Theorem \ref{th-main} it remains to show that the assumptions on $c_1,$ $c_2$ and $c_3$ are not restrictive. Actually, by an affine conjugation of $q$ we can always assume that $c_1+c_2+c_3=0.$ The fact that they are not aligned is probably true for a generic choice of $q.$ However, we only prove the following result which is sufficient for our purpose.
\begin{lemma}\label{le-bon-point-rep}
Let $q$ be in $\Pc_d.$ If $t_0=3$ and $m_0\geq2$ then $q$ has three repelling points, $r_1,$ $r_2$ and $r_3,$ of period $m_1:=m_0t_0$ outside its postcritical set $P(q)$ such that the numbers $c_1,$ $c_2,$ $c_3$ defined by \eqref{eq-ci} form a non degenerated equilateral triangle with $c_1+c_2+c_3=0.$
\end{lemma}
\begin{proof}
Let $q$ be in $\Pc_d.$ Define $c(w):=\sum_{i=0}^{m_1-1}b_i(\lambda_0)q^i(w).$ If $r$ is a periodic point with period $m_1$ then using the fact that $b_{i-m_0}(\lambda_0)=\rho(\lambda_0)b_i(\lambda_0)$ we obtain
$$c(q^{m_0}(r))=\sum_{i=0}^{m_1-1}b_i(\lambda_0)q^{i+m_0}(r)=\sum_{i=0}^{m_1-1}b_{i-m_0}(\lambda_0)q^i(r)=\rho(\lambda_0)c(r),$$
and thus $c(q^{2m_0}(r))=\rho(\lambda_0)^2c(r).$ Moreover, $\rho(\lambda_0)$ is a primitive $3$th root of unity thus if $c(r)\neq0$ then $c_j:=c(r_j),$ $j\in\{1,2,3\}$ with $r_1:=r,$ $r_2:=q^{m_0}(r),$ $r_3:=q^{2m_0}(r),$ form an equilateral triangle with $c_1+c_2+c_3=0.$

It remains to show that there exists a repelling point $r$ of period $m_1$ for $q$ such that $c(r)\neq0$ and $r\notin P(q).$ This simply follows from the fact that $c(w)$ is a polynomial of degree $d^{m_1-1}$ (since $b_i(\lambda_0)\neq0$ for $i\in\Zb$) and that, by Fatou's theorem, there exist at most $3m_1(d-1)$ periodic points of period $m_1$ counted with multiplicity which are non repelling or contained in $P(q).$ As there are $d^{m_1}$ points of period $m_1$ and that $d^{m_1}-3m_1(d-1)>d^{m_1-1}$ when $d\geq2$ and $m_1=3m_0\geq6,$ there exists at least one repelling point $r$ of period $m_1$ such that $c(r)\neq0$ and $r\notin P(q).$
\end{proof}

\begin{proof}[Proof of Theorem \ref{th-main}]
Let $p$ and $q$ be two elements of $\Pc_d$ such that $p$ belongs to the bifurcation locus. Therefore, $p$ can be approximated by polynomials having a cycle of multiplier $e^{2i\pi/3}.$ As a consequence, we can obtain a family $(p_\lambda)_{\lambda\in M}$ of polynomials of degree $d$ which has a persistent periodic point $z_0$ which is not fixed, with a non-constant multiplier $\rho(\lambda)$ and such that there is $\lambda_0\in M$ with $\rho(\lambda_0)=e^{2i\pi/3}$ and $p_{\lambda_0}$ is arbitrarily close to $p.$ As $\rho(\lambda)$ is not constant, there exists $(\lambda_n)_{n\geq1}$ converging to $\lambda_0$ such that $|\rho(\lambda_n)|>1$ (resp. $|\rho(\lambda_n)|<1$). By conjugating the whole family and $p$ by $z\mapsto z+z_0,$ we can assume that $z_0=0.$ Therefore, the combination of Lemma \ref{le-bon-point-rep} and Theorem \ref{th-constru} gives a sequence $(f_n)_{n\geq1}$ of the form $f_n(z,w):=(p_{\lambda_n}(z)+(\beta_n/2)w,q(w))$ and such that some iterates of $f_n$ have a blender of repelling type (resp. saddle type). The result follows since $(f_n)_{n\geq1}$ converges to $(p_{\lambda_0},q)$ and $p_{\lambda_0}$ was arbitrarily close to $p.$
\end{proof}

We conclude this section with a complement about the repelling case in Theorem \ref{th-constru} which will be useful in Section \ref{sec-ouvert-bif}. The set $\widetilde\Lambda_{n,\alpha}$ defined below will allow us to apply Proposition \ref{prop-ouvert-de-bif} and the existence of the second blender $\Lambda'_{n,\alpha}$ is a minor technical point (see Lemma \ref{le-lambda}).
\begin{proposition}\label{prop-basic}
In the repelling case of Theorem \ref{th-constru}, there exists a second blender $\Lambda'_{n,\alpha}$ with
$$\Lambda'_{n,\alpha}=\bigcap_{i\geq0}f^{-im_1l_n}_{n,\alpha}\left(\overline{\bigcup_{j=1}^3\left(\frac{|\alpha\beta_n|}{3}H_{j}\right)\times V^{l_n}_j}\right), \Lambda_{n,\alpha}=\bigcap_{i\geq0}f^{-im_1l_n}_{n\alpha}\left(\overline{\bigcup_{j=1}^3(|\alpha\beta_n|H_{j})\times V^{l_n}_j}\right),$$
for $\alpha\in\Db^*.$ Here, $H_j$ denote the sets given in Lemma \ref{le-triangle}. Moreover, by reducing $\beta_n$ if necessary, for each $n\geq1$ there are $R_n>0$ and a neighborhood $\Uc_n\subset\Hc_d(\Pb^2)$ of $f_{n,0}$ which contains the family $(f_{n,\alpha})_{\alpha\in\Db}$ and such that $R_n>|\beta_n|$ and for all $f\in\Uc_n$
$$\widetilde\Lambda_f:=\bigcap_{i\geq0}f^{-im_1l_n}\left(\overline{\bigcup_{j=1}^3\left(R_n\Db\right)\times V^{l_n}_j}\right)$$
defines a basic repeller for $f^{m_1l_n}$ which moves holomorphically for $f\in\Uc_n$ and which can only intersect $P(f)$ properly.
\end{proposition}
\begin{proof}
The first point indeed could be immediately deduce from Proposition \ref{prop-exi-rep} at the end of the proof of Theorem \ref{th-constru}.

The last point is a simple observation. Since $0$ is a repelling periodic point for $p_{\lambda_n}$ there exists $R_n>0$ such that $p^{m_0}_{\lambda_n}$ is conjugated to $z\mapsto\rho(\lambda_n)z$ on $D(0,R_n).$ Hence, the set $\widetilde\Lambda_{n,0}$ defined as above is a repelling hyperbolic set contained in $\{0\}\times\Cb$ and homeomorphic to a Cantor set. Its dynamics is conjugated to a shift of three symbols and is therefore transitive, i.e. $\widetilde\Lambda_{n,0}$ is a basic set for $f^{m_1l_n}_{n,0}.$ Moreover, by Theorem \ref{th-motion} it can be followed holomorphically as a basic repeller in a small neighborhood $\Uc_n$ of $f_{n,0}=(p_{\lambda_n},q).$ Since there exist polynomials arbitrarily close to $p_{\lambda_n}$ such that $0$ is periodic but not postcritical, an intersection between $\widetilde\Lambda_f$ and $P(f)$ is necessarily proper. And if $|\beta_n|$ is chosen small enough then $f_{n,\alpha}\in\Uc_n$ for all $\alpha\in\Db.$
\end{proof}

\section{From repelling blenders to robust bifurcations}\label{sec-ouvert-bif}
A large part of this section is inspired by \cite{dujardin-bif} (see Lemma \ref{le-para} and Lemma \ref{le-point-rep-pc} below). We show that the blenders of repelling type constructed in Section \ref{sec-cons} give rise to open sets in the bifurcation locus of $\Hc_d(\Pb^2)$ arbitrarily close to a bifurcation in $\Pc_d\times\Pc_d.$

Let $p$ and $q$ be two elements of $\Pc_d$ such that $p$ belongs to the bifurcation locus. As in the proof of Theorem \ref{th-main}, we choose a polynomial $p_{\lambda_0}$ close to $p$ which has a periodic (and not fixed) point of multiplier $e^{2i\pi/3}.$ By conjugation, we can assume that this point is $0.$ The following result says that $p_{\lambda_0}$ can be approximated by maps such that $0$ is a Misiurewicz point. Its proof is contained in the proof of \cite[Corollary 4.11]{dujardin-bif}.

\begin{lemma}\label{le-para}
If $0$ is a parabolic periodic point for a polynomial $p_{\lambda_0}$ then there exist a family of polynomials $(p_\lambda)_{\lambda\in M}$ and a sequence $(\lambda_n)_{n\geq1}$ of $M$ converging to $\lambda_0$ such that for each $n\geq1$ $0$ is a repelling periodic point for $p_{\lambda_n}$ which belongs to the postcritical set $P(p_{\lambda_n}).$
\end{lemma}

By Lemma \ref{le-bon-point-rep} and Theorem \ref{th-constru}, there exist sequences $(l_n)_{n\geq1}$ and $(\beta_n)_{n\geq1},$ which are chosen in order to satisfy Proposition \ref{prop-basic}, such that if we fix $n\geq1$ large enough then for $\alpha\in\Db^*$ the $(m_1l_n)$-th iterate of the map
$$f_{n,\alpha}(z,w):=(p_{\lambda_n}(z)+\alpha\beta_nw,q(w))$$
satisfies Proposition \ref{prop-inter-robuste} for some constant $\delta_\alpha>0$ and sets $H_{j,\alpha}$ and $V_j,$ $j\in\{1,2,3\}.$ Actually, the map $f_{n,\alpha}$ is a rescaling by $\alpha\beta_n$ of the map $g_n$ used in the proof of Theorem \ref{th-constru} thus $\delta_\alpha=|\alpha\beta_n|\delta_0$ and $H_{j,\alpha}$ is contained in $D(0,|\alpha\beta_n|)$ and is the rescale by $\alpha\beta_n$ of a set $H_j\subset\Db.$ And if $R_n>|\beta_n|,$ $\Uc_n\subset\Hc_d(\Pb^2)$ are the objects given by Proposition \ref{prop-basic}, we define for $\alpha$ in $\Db^*$ and in $\Db$ respectively
$$\Lambda_\alpha:=\bigcap_{i\geq0}f_{n,\alpha}^{-im_1l_n}\left(\overline{\cup_{j=1}^3H_{j,\alpha}\times V_j}\right)\ \text{ and }\ \widetilde\Lambda_\alpha:=\bigcap_{i\geq0}f_{n,\alpha}^{-im_1l_n}\left((R_n\Db)\times\cup_{j=1}^3V_j\right).$$
By Proposition \ref{prop-basic} the set $\widetilde\Lambda_\alpha$ is a basic repeller which can be followed holomorphically for $f\in\Uc_n$ thus in particular for $f_{n,\alpha},$ $\alpha\in\Db.$ And the set $\Lambda_\alpha$ has the following properties.
\begin{lemma}\label{le-lambda}
For each $\alpha_1\in\Db^*$ there is a neighborhood $B\subset\Db^*$ such that there exists $x_{\alpha_1}$ in $\Lambda_{\alpha_1}$ whose continuation as a point of $\widetilde\Lambda_{\alpha_1}$ belongs to $\Lambda_\alpha$ for $\alpha\in B$ and is of the form $x_\alpha=(z_\alpha,w_0)$ where $w_0$ is outside the postcritical set $P(q)$ of $q.$
\end{lemma}
\begin{proof}
As observe in the proof of Proposition \ref{prop-basic}, $\widetilde\Lambda_0$ is contained in $\{0\}\times\Cb$ and can be identify to $\cap_{i\geq0}q^{-im_1l_n}(\cup_{j=1}^3V_j)$ which is a Cantor set independent of $\alpha.$ As the holomorphic motion of $\widetilde\Lambda_\alpha$ preserves the dynamics and $f_{n,\alpha}$ is a skew product with the same second coordinate that $f_{n,0}$, it follows that for all $\alpha\in\Db$ the projection on the second coordinate of the set $\widetilde\Lambda_\alpha$ is a bijection on this Cantor set. Therefore, the second coordinate of the holomorphic continuation of a point in $\widetilde\Lambda_\alpha$ is constant.

Now, let $\alpha_1$ be in $\Db^*.$ By Proposition \ref{prop-basic} the set
$$\Lambda'_{\alpha_1}:=\bigcap_{i\geq0}f_{n,\alpha_1}^{-im_1l_n}\left(\overline{\cup_{j=1}^3(1/3H_{j,\alpha_1})\times V_j}\right)$$
is a blender. As we had observed after Proposition \ref{prop-inter-robuste}, it implies that the projection on the first coordinate of $\Lambda'_{\alpha_1}$ contains a non-empty open set. In particular, it is uncountable and thus there is $x_{\alpha_1}=(z_{\alpha_1},w_0)\in\Lambda'_{\alpha_1}$ such that $w_0$ doesn't belong to $P(q).$ Finally, since by construction $H_{j,\alpha_1}$ satisfies $\overline{1/3H_{j,\alpha_1}}\subset D(0,\frac{|\alpha_1\beta_n|}{2})\subset H_{j,\alpha_1}$ and since the holomorphic motion is continuous and is compatible with the dynamics, we obtain that in a small neighborhood of $\alpha_1$ the continuation $x_\alpha$ of $x_{\alpha_1}$ verifies $f_{n,\alpha}^{im_1l_n}(x_\alpha)\in\cup_{j=1}^3H_{j,\alpha}\times V_j$ for all $i\geq0,$ i.e. $x_\alpha\in\Lambda_\alpha.$
\end{proof}

The critical set $C(f_{n,\alpha})$ of $f_{n,\alpha}$ is independent of $\alpha$ and is equal to $C(p_{\lambda_n})\times\Cb\cup\Cb\times C(q).$ A key point in what follows is that for some parameter the images of a vertical irreducible component of this set intersects the blender. Following \cite{dujardin-bif}, we use a result contained in \cite{bbd-bif} to obtain such parameter.
\begin{proposition}[{\cite[Proposition 2.2]{dujardin-bif}}]\label{prop-bbd}
Let $(f_\lambda)_{\lambda\in M}$ be a holomorphic family of degree $d$ endomorphism of $\Pb^k$ and let $\gamma\colon M\to\Pb^k$ be a holomorphic map. If $\lambda_0\in M$ belongs to the bifurcation locus then there exists $\lambda_1\in M$ arbitrarily close to $\lambda_0$ such that $\gamma(\lambda_1)\in P(f_{\lambda_1}).$
\end{proposition}
\begin{lemma}\label{le-point-rep-pc}
There exist $\alpha_0\in\Db^*$ and $x_{\alpha_0}\in\Lambda_{\alpha_0}$ such that $x_{\alpha_0}\in\cup_{i\geq0}f^i_{n,\alpha_0}(C(p_{\lambda_n})\times\Cb)$ and $x_{\alpha_0}\notin\cup_{i\geq0}f^i_{n,\alpha_0}(\Cb\times C(q)).$
\end{lemma}
\begin{proof}
We consider two cases. First, if the bifurcation locus of the family $(f_{n,\alpha})_{\alpha\in\Db}$ is not empty then there exists $\alpha_1\in\Db^*$ belonging to it since it is a perfect set of $\Db.$ By Lemma \ref{le-lambda}, there exists a point $x_{\alpha_1}=(z_{\alpha_1},w_0)$ in $\Lambda_{\alpha_1}$ such that $w_0$ is not in the postcritical set $P(q)$ and its holomorphic continuation $x_\alpha$ satisfies $x_\alpha\in\Lambda_\alpha$ for $\alpha$ in a small neighborhood of $\alpha_1.$ Hence, Proposition \ref{prop-bbd} implies that there exists $\alpha_0$ arbitrarily close to $\alpha_1$ such that $x_{\alpha_0}=(z_{\alpha_0},w_0)$ is in the postcritical set of $f_{n,\alpha_0}$ i.e. $x_{\alpha_0}\in\cup_{i\geq0}f^i_{n,\alpha_0}(C(f_{n,\alpha_0})).$ As $w_0$ is not in $P(q)$ it follows that $x_{\alpha_0}\in\cup_{i\geq0}f^i_{n,\alpha_0}(C(p_{\lambda_n})\times\Cb)$ and $x_{\alpha_0}\notin\cup_{i\geq0}f^i_{n,\alpha_0}(\Cb\times C(q)).$

We now assume that the family $(f_{n,\alpha})_{\alpha\in\Db}$ is stable. In particular, we can follow holomorphically each repelling point for $\alpha\in\Db.$ We choose $\alpha_0\in\Db^*,$ a repelling periodic point $r_{\alpha_0}=(z_{\alpha_0},w)$ in $\widetilde\Lambda_{\alpha_0}$ and we denote by $r_{\alpha}=(z_{\alpha},w)$ its continuation. Since $\widetilde\Lambda_\alpha$ moves holomorphically for $\alpha\in\Db$ and $\widetilde\Lambda_0\subset\{0\}\times\Cb$ we have $z_0=0.$ On the other hand, as $0$ is postcritical for $p_{\lambda_n}$ there are $z_1\in C(p_{\lambda_n})$ and $i\geq1$ such that $f^i_{n,0}(\{z_1\}\times\Cb)=\{0\}\times\Cb.$ We use a second time the fact that the family is stable to obtain that $r_\alpha\in f_{n,\alpha}^i(\{z_1\}\times\Cb).$ Since $\Lambda_\alpha\subset\widetilde\Lambda_\alpha$ and the repelling points are dense in $\widetilde\Lambda_\alpha$ this implies that $\Lambda_\alpha\subset f_{n,\alpha}^i(\{z_1\}\times\Cb).$ We then conclude as above by taking a point $x_{\alpha_0}=(z_0,w_0)\in\Lambda_{\alpha_0}$ with $w_0\notin P(q).$
\end{proof}

\begin{lemma}
If $\alpha_0$ is as in Lemma \ref{le-point-rep-pc} then the postcritical set $P(f_{n,\alpha_0})$ contains a vertical graph tangent to $C_{\delta_{\alpha_0}}$ in $H_{j,\alpha_0}\times V_j$ for some $j\in\{1,2,3\}.$
\end{lemma}
\begin{proof}
Let $\alpha_0\in\Db^*$ and let $x_{\alpha_0}=(z_{\alpha_0},w_0)$ be a point in $\Lambda_{\alpha_0}$ for $f_{n,\alpha_0}$ such that $x_{\alpha_0}\in\cup_{i\geq0}f^i_{n,\alpha_0}(C(p_{\lambda_n})\times\Cb)$ and $x_{\alpha_0}\notin\cup_{i\geq0}f^i_{n,\alpha_0}(\Cb\times C(q)).$ There exist $z_1\in C(p_{\lambda_n}),$ $i_0\in\{1,2,3\}$ and $\widetilde l\geq0$ minimal such that $x_{\alpha_0}$ belongs to $H_{i_0,\alpha_0}\times V_{i_0}$ and to $X:=f_{n,\alpha_0}^{\widetilde l}(\{z_1\}\times\Cb).$ We claim that the image of $X$ by a large iterate of $f^{l_nm_1}_{n,\alpha_0}$ contains a vertical graph tangent to $C_{\delta_{\alpha_0}}.$ Indeed, since $w_0\notin P(q)$ and $\widetilde l$ is minimal $X$ is smooth at $x_{\alpha_0}$ and cannot be tangent to the horizontal direction. Therefore, $X$ contains a vertical graph over a small neighborhood $U$ of $w_0.$ If $(i_k)_{k\geq0}$ denotes the sequence in $\{1,2,3\}$ such that $q^{kl_nm_1}(w_0)\in V_{i_k}$ then for $k$ large enough the set $\psi_{i_0i_1}\circ\cdots\circ\psi_{i_{k-1}i_k}(V_{i_k})$ is contained in $U.$ Here, $\psi_{ij}\colon V_{j}\to W_{ij}\subset V_{i}$ is the inverse branch of $q^{lm_1}$ given by Lemma \ref{le-vois-rep}. Hence, $f^{kl_nm_1}_{n,\alpha_0}(X)$ contains a vertical graph over $V_{i_k}.$ Moreover, by the estimates on $D_{(z,w)}g_n^{m_1l_n}$ obtained in the proof of Theorem \ref{th-constru}, if $k$ is large enough then this graph has to be tangent to $C_{\delta_{\alpha_0}}$ and its image is contained in $H_{i_k,\alpha_0}\times V_{i_k}.$ This conclude the proof.
\end{proof}

\begin{proof}[Proof of Corollary \ref{coro-bif}]
Let $(p,q)$ be in $\Pc_d\times\Pc_d.$ If $p$ is in the bifurcation locus of $\Pc_d$ then as we have seen above, there exist a polynomial $p_{\lambda_n}$ close to $p,$ a number $\beta_n\in\Cb$ close to $0$ and $\alpha_0\in\Db^*$ such that the postcritical set of the map $f_{n,\alpha_0}(z,w):=(p_{\lambda_n}(z)+\alpha_0\beta_nw,q(w))$ contains a vertical graph tangent to $C_{\delta_{\alpha_0}}$ in $H_{j,\alpha_0}\times V_j$ for some $j\in\{1,2,3\}.$ Moreover, a large iterate of $f_{n,\alpha_0}$ possesses a blender $\Lambda_{\alpha_0}$ contained in the basic set $\widetilde\Lambda_{\alpha_0}$ and satisfies Proposition \ref{prop-inter-robuste} with the constant $\delta_{\alpha_0}$ and the sets $H_{j,\alpha_0}$ and $V_j.$ Therefore, if $f\in\Uc_n$ is close enough to $f_{n,\alpha_0}$ then the continuation of $\widetilde\Lambda_{\alpha_0}$ for $f$ intersects its postcritical set. By Proposition \ref{prop-basic} this intersection is proper. Hence, by Proposition \ref{prop-ouvert-de-bif} $f_{n,\alpha_0}$ is in the interior of the bifurcation locus $\bif(\Hc_d(\Pb^2))$ and thus $(p,q)$ is in the closure of the interior of $\bif(\Hc_d(\Pb^2)).$
\end{proof}

\section{Fat attractors}\label{sec-att}
In this section, we prove Corollary \ref{coro-as} and Theorem \ref{th-att}. The former is a direct consequence of Theorem \ref{th-constru} and the latter follows from a construction based on Proposition \ref{prop-exi-selle}. An elementary fact about attracting set that we will use several times is that if $U$ is an open subset of $\Pb^2$ and $f$ is a rational map of $\Pb^2$ such that $f(\overline{U})\subset U$ then $g(\overline{U})\subset U$ for all small perturbations $g$ of $f,$ i.e. $U$ is also a trapping region for $g.$
\begin{proof}[Proof of Corollary \ref{coro-as}]
Let $(p,q)$ be in $\Pc_d\times\Pc_d.$ Exactly as in the proof of Corollary \ref{coro-bif}, if $p$ is in the bifurcation locus of $\Pc_d$ then arbitrarily close to it and $0$ respectively there exist $p_{\lambda_n}$ and $\beta_n\in\Cb$ such that for all $\alpha\in\Db^*$ the $m_1l_n$-th iterate of $f_{n,\alpha}(z,w):=(p_{\lambda_n}(z)+\alpha\beta_nw,q(w))$ has a blender of saddle type. To be more precise, there are three open sets $V_1,$ $V_2$ and $V_3$ in $\Cb$ such that $\overline{\cup_{j=1}^3D(0,|\alpha\beta_n|)\times V_j}\subset f_{n,\alpha}^{m_1l_n}(\cup_{j=1}^3D(0,|\alpha\beta_n|)\times V_j).$ On the other hand, $0$ is an attracting fixed point for $p_{\lambda_n}^{m_1l_n}$ and it is easy to check that the line $\{z=0\}$ is an attracting set in $\Pb^2$ for $f_{n,0}^{m_1l_n}.$ Hence, it admits a trapping region $U$ and if $\alpha\in\Db^*$ is close enough to $0$ then $\cup_{j=1}^3D(0,|\alpha\beta_n|)\times V_j\subset U$ and $\overline{f_{n,\alpha}^{m_1l_n}(U)}\subset U,$ i.e. $f_{n,\alpha}$ has an attracting set containing $\cup_{j=1}^3D(0,|\alpha\beta_n|)\times V_j.$ Since the inclusion 
$$\overline{\cup_{j=1}^3D(0,|\alpha\beta_n|)\times V_j}\subset f_{n,\alpha}^{m_1l_n}(\cup_{j=1}^3D(0,|\alpha\beta_n|)\times V_j)$$
is stable under small perturbations of $f_{n,\alpha}$ it follows that $(p,q)$ is in the closure of the interior of endomorphisms possessing a proper attracting set with non-empty interior.
\end{proof}
All these examples are not transitive as they all possess an attracting point near $[0:1:0].$ However, as we will see it turns out that the composition of an automorphism and of two such maps can have an attractor with non-empty interior.

The remaining part of this section is devoted to the proof of Theorem \ref{th-att} which splits into the three elementary lemmas below. We don't try to obtain the existence of attractors with non-empty interior in a general setting. For simplicity, we only use perturbations of iterates of the map $(z,w)\mapsto(z,q(w))$ where $q(w)=w^4.$ The ideas are the following. Even if $q$ is far from being transitive on $\Pb^1,$ its postcritical set is particularly simple. Hence, if we consider the automorphism of $\Pb^1$ given by $\psi(w)=\frac{iw+1}{w+i}$ it is easy to check that for all integers $a,b\geq1$ each critical point of the map $q^a\circ\psi\circ q^b$ is eventually mapped to the fixed repelling point $1.$ Therefore, its Fatou set is empty and this map is transitive on $\Pb^1.$ In what follows, we obtain an endomorphism $f$ with an attracting set $A$ containing a blender of saddle type. The interior of the unstable set of this blender contains an algebraic curve which is thus included in the interior of $A.$ On the other hand, $f$ preserves a pencil of lines $\Pc$ and its action on $\Pc$ is given by $q^a\circ\psi\circ q^b$ for some $a,b\geq1.$ Hence, arguments going back to \cite{jonsson-weickert-attractor} (see also \cite{fs-example}) imply that $f_{|A}$ is transitive, i.e. $A$ is an attractor.

The first lemma is a variation on Lemma \ref{le-vois-rep} and Proposition \ref{prop-exi-selle} where we use the three repelling fixed points of $q,$ $r_1:=1,$ $r_2:=e^{2i\pi/3}$ and $r_3:=e^{4i\pi/3}.$

\begin{lemma}\label{le-att1}
Let $R>0.$ There exist $0<\lambda<1,$ $\epsilon_0>0,$ $l_0\geq1$ and a family of neighborhoods $(U^l_i)_{l\geq l_0}$ of $r_i$ such that the diameter of $U_i^l$ converges exponentially fast to $0$ with $l$ and for each $l\geq l_0$ and $\alpha\in\Cb^*$ the map $g_\alpha(z,w):=(\lambda z+\alpha w,q^l(w))$ satisfies
$$\overline{(\alpha\epsilon_0^{-1}\Db)\times(R\Db\setminus R^{-1}\Db)}\subset g_\alpha((\alpha\epsilon_0^{-1}\Db)\times\cup_{i=1}^3U_i^l).$$
\end{lemma}
\begin{proof}
We consider the open sets $V_i^l$ given by Lemma \ref{le-vois-rep} applied to $q$ and $r_i.$ As $q(w)=w^4$ obviously there exists $l_1\geq1$ such that $\overline{R\Db\setminus R^{-1}\Db}\subset\cap_{i=1}^3q^{l_1}(V^0_i).$ Then for $l\geq l_1$ we define $U_i^l:=V_i^{l-l_1}$ and we obtain
$\overline{R\Db\setminus R^{-1}\Db}\subset q^{l_1}(V_i^0)\subset q^l(U_i^l).$ Since the diameter of $V^l_i$ converges exponentially fast to $0$ with $l,$ the same holds for $U_i^l.$

On the other hand, as $r_1,$ $r_2$ and $r_3$ are three non-aligned points in $\Cb$ such that $r_1+r_2+r_3=0,$ by Lemma \ref{le-triangle} there exist $0<\lambda<1$ and $\epsilon_0>0$ such that $\overline{\Db}\subset\cup_{i=1}^3\phi_i(\Db)$ where $\phi_i(z)=\lambda z+\epsilon_0r_i.$ Moreover, if $l_0\geq l_1$ is large enough and $l\geq l_0$ then the sets $U_i^l$ are arbitrarily small. Hence, if $l\geq l_0$ then the map $g(z,w):=(\lambda z+\epsilon_0w,q^l(w))$ is close to $(z,w)\mapsto(\phi_i(z),q^l(w))$ on $\Db\times U^l_i.$ Thus, using that $\overline{R\Db\setminus R^{-1}\Db}\subset q^l(U_i^l),$ it follows exactly as in the proof of Proposition \ref{prop-exi-rep} that 
$$\overline{\Db\times(R\Db\setminus R^{-1}\Db)}\subset g(\Db\times\cup_{i=1}^3U_i^l).$$
It also implies the general case with $\alpha\in\Cb^*$ since $g_\alpha$ is the conjugation of $g$ by $(z,w)\mapsto(\alpha\epsilon_0^{-1}z,w).$
\end{proof}
Recall that $\psi(w)=\frac{iw+1}{w+i}.$ From now on, let $R>0$ be large enough in order to get $\psi^{-1}(\{r_1,r_2,r_3\})\subset R\Db\setminus R^{-1}\Db.$ We choose $l\geq l_0$ such that $\overline{\cup_{i=1}^3(U_i^l\cup\psi^{-1}(U_i^l))}\subset R\Db\setminus R^{-1}\Db$ and for $(\alpha,\eta)\in\Cb^2$ we define
$$G_{\alpha,\eta}(z,w):=(\lambda z+\alpha w+\eta q^l(z),q^l(w)).$$
It can be seen as a rational map of $\Pb^2$ and if $\eta\neq0$ then $G_{\alpha,\eta}$ extends to an endomorphism of $\Pb^2$ which we still denote by $G_{\alpha,\eta}.$ Moreover, the map $\psi$ extends to an automorphism $\Psi[z:w:t]:=[z:iw+t:it+w]$ which acts by $\psi$ on the line $X:=\{[z:w:t]\in\Pb^2\,|\, z=0\}.$

To conclude the proof of Theorem \ref{th-att}, we need the two following results.
\begin{lemma}\label{le-att2}
There exist $\widetilde l\geq1$ and $\widetilde\lambda>1$ such that for each $\alpha\in\Cb^*$ if $\eta\in\Cb^*$ sufficiently close to $0$ then the map $F_{\eta}(z,w):=(\widetilde\lambda z+\eta q^{\widetilde l}(z),q^{\widetilde l}(w))$ satisfies
$$X\cup\overline{(\alpha\epsilon_0^{-1}\Db)\times\cup_{i=1}^3U_i^l)}\subset F_{\eta}\circ\Psi((\alpha\epsilon_0^{-1}\Db)\times(R\Db\setminus R^{-1}\Db)).$$
\end{lemma}
\begin{proof}
Let $\widetilde l\geq1$ be such that $q^{\widetilde l}\circ\psi(R\Db\setminus R^{-1}\Db)=\Pb^1.$ Such $\widetilde l$ exists since the complement in $\Pb^1$ of $\psi(R\Db\setminus R^{-1}\Db)$ consists of two balls containing $i$ and $-i$ and disjointed from $\Rb.$ Hence $X\subset F_{\eta}\circ\Psi((\alpha\epsilon_0^{-1}\Db)\times(R\Db\setminus R^{-1}\Db))$ for all $(\alpha,\eta)\in(\Cb^*)^2.$

For the second part of the inclusion it is sufficient to choose $\widetilde\lambda>1$ large enough in order to compensated the contraction on the $z$ coordinate due to $\Psi$ on $\Cb\times\psi^{-1}(\cup_{i=1}^3U_i^l).$ Then $\overline{(\alpha\epsilon_0^{-1}\Db)\times\cup_{i=1}^3U_i^l)}\subset F_{0}\circ\Psi((\alpha\epsilon_0^{-1}\Db)\times(R\Db\setminus R^{-1}\Db))$ for all $\alpha\in\Cb^*.$ Thus, if $\alpha\in\Cb^*$ is fixed then the same inclusion holds for small $\eta\in\Cb^*.$
\end{proof}

\begin{lemma}\label{le-att3}
There exist $\beta_1>0,$ $\rho>0$ and $N\geq1$ such that for all $(\alpha,\eta)\in(\beta_1\Db^*)^2$ the open subset of $\Pb^2$ given by $U_\rho:=\{[z:w:t]\in\Pb^2\,|\, |z|<\rho\max(|w|,|t|)\}$ is a trapping region for $F_\eta\circ\Psi\circ(G_{\alpha,\eta}^N).$
\end{lemma}
\begin{proof}
Observe that for all $\rho>0,$ $\overline{G_{0,0}(U_\rho)}\subset U_\rho$ and $\bigcap_{N\geq0}G^N_{0,0}(U_\rho)=X.$ Moreover, if $\rho'>0$ then $\overline{F_0\circ\Psi(U_{\rho'})}$ is contained in $U_\rho$ for some $\rho>0.$ Hence, there exists $N\geq1$ such that $G^N_{0,0}(U_\rho)\subset U_{\rho'}$ and thus
$$\overline{F_0\circ\Psi\circ(G_{0,0}^N)(U_\rho)}\subset\overline{F_0\circ\Psi(U_{\rho'})}\subset U_\rho,$$
i.e. $U_\rho$ is a trapping region for $F_0\circ\Psi\circ G^N_{0,0}.$ The result follows since this property is stable under small perturbations.
\end{proof}

\begin{proof}[Proof of Theorem \ref{th-att}]
Using the notations of the above lemmas define $Z_\alpha:=(\alpha\epsilon_0^{-1}\Db)\times\cup_{i=1}^3U_i^l$ and let $\alpha\in\beta_1\Db^*$ be small enough such that $Z_\alpha\subset U_\rho.$ By Lemma \ref{le-att1} and Lemma \ref{le-att2}, if $\eta\in\beta_1\Db^*$ is close enough to $0$ then  $X\cup Z_\alpha\subset F_\eta\circ\Psi((\alpha\epsilon_0^{-1}\Db)\times(R\Db\setminus R^{-1}\Db))$ and $(\alpha\epsilon_0^{-1}\Db)\times(R\Db\setminus R^{-1}\Db)\subset G_{\alpha,\eta}(Z_\alpha)$ which implies $Z_\alpha\subset G_{\alpha,\eta}(Z_\alpha).$ Therefore, $Z_\alpha\subset G_{\alpha,\eta}^{N-1}(Z_\alpha)$ and thus if $f:=F_\eta\circ\Psi\circ G^N_{\alpha,\eta}$ then $X\cup Z_\alpha\subset f(Z_\alpha).$ Moreover, since $(\alpha,\eta)\in(\beta_1\Db^*)^2$ the set $U_\rho$ is a trapping region for $f$ which contains $Z_\alpha.$ Hence, $A:=\bigcap_{n\geq0}f^n(U_\rho)$ contains $Z_\alpha$ thus $f(Z_\alpha)$ which is an open set which itself contains the line $X.$

It remains to show that $A$ is an attractor. But, obverse that both the maps $G_{\alpha,\eta}$ and $F_\eta\circ\Psi$ preserve the pencil $\Pc$ of lines passing through the point $[1:0:0].$ With the standard identification between $\Pc$ and $\Pb^1,$ their actions on $\Pc$ are given by $q^l$ and $q^{\widetilde l}\circ\psi$ respectively. Hence, the one of $f$ is given by $q^{\widetilde l}\circ\psi\circ q^{lN}$ which is a transitive rational map on $\Pb^1.$ From this, it follows as in the proof of \cite[Lemma 4]{jonsson-weickert-attractor} (see also \cite{fs-example}) that $f_{|A}$ is topologically mixing and thus transitive.
\end{proof}
\begin{remark}
To such an attracting set Dinh \cite{d-attractor} (see also \cite{t-attractor}) associates an attracting current and an equilibrium measure. Observe that since in all the examples above a pencil of line is preserved, the support of the attracting current is exactly equal to the corresponding attracting set. This contrasts with the examples from the Section \ref{sec-henon}. Another observation is that, in an obvious way using product maps, the bifurcations in $\Pc_d\times\Pc_d$ can be detected by the smallest Lyapunov exponents of the equilibrium measures associated to attracting sets.
\end{remark}
\section{Cycles of blenders}\label{sec-cycle}
The goal of this section is to explain how one can perturb product maps of the form $(z,w)\mapsto(z,q(w))$ in order to obtain a \textit{cycle} between a blender of repelling type $\Lambda_r$ and one of saddle type $\Lambda_s.$ The fact that two hyperbolic sets form a cycle simply means that the unstable manifold of one set intersects the stable manifold of the other one and vice versa. We will see that in our setting it will be sufficient to prove that the unstable manifold of $\Lambda_s$ intersects $\Lambda_r.$ The proof of Theorem \ref{th-cycle} will then follow easily form this construction using classical results in hyperbolic and complex dynamics.

We are able to fulfill the construction below in a more general setting, but not as general as the one of Theorem \ref{th-main}. So, for simplicity we restrict ourselves to the following context. We first define $q(w):=w^7,$ $r_1:=1,$ $r_2:=e^{2\pi/3},$ $r_3:=e^{4\pi/3}$ and for $i\in\{1,2,3\},$ $s_i:=-r_i.$ Observe that these six points are repelling fixed points of $q$ and $r_1+r_2+r_3=s_1+s_2+s_3=0.$ We denote by $H_i$ the sets defined in Lemma \ref{le-triangle} with $c_i=r_i$ and $V_i^l$ (resp. $U_i^l$) the open sets associated to $r_1,$ $r_2$ and $r_3$ (resp. $s_1,$ $s_2$ and $s_3$) by Lemma \ref{le-vois-rep} with $m_1=1.$ Observe that we can easily require in addition that $\overline{\cup_{i=1}^3V_i^l}\subset q^l(U_j^l)$ for all $j\in\{1,2,3\}.$ The idea of the following proposition is to combine Proposition \ref{prop-exi-rep} and Proposition \ref{prop-exi-selle} using the fact that the polynomial $(w^3+1)/2$ equals $1$ for $w=r_i$ and vanishes for $w=s_i.$ The condition $1/10>\alpha_1>\alpha_2>0$ will simplify the end of the proof of Theorem \ref{th-cycle-bis}.
\begin{proposition}\label{prop-double-blender}
There exist $1/10>\alpha_1>\alpha_2>0,$ $\epsilon_1,\epsilon_2, R>0$ and $l_0\geq1$ such that for all $l\geq l_0$ the map
$$F_l(z,w):=(z+(\alpha_1z+\epsilon_1w)(w^3+1)/2-\alpha_2z-\epsilon_2w,q^l(w))$$
has a blender of repelling type $\Lambda^l_r$ in $\Db\times\cup_{i=1}^3V^l_i$, a blender of saddle type $\Lambda^l_s$ in $4^{-1}\Db\times\cup_{i=1}^3U^l_i$ and for each $\widehat x$ in the natural extension of $\Lambda^l_s$ the unstable manifold $W^u_{\widehat x}$ intersects $\Lambda^l_r.$ Moreover, the sets
$$\widetilde\Lambda^l_r:=\bigcap_{n\in\Zb}F_l^n\left(R\Db\times\cup_{i=1}^3V_i^l\right)\ \text{ and }\ \widetilde\Lambda^l_s:=\bigcap_{n\in\Zb}F_l^n\left(R\Db\times\cup_{i=1}^3U_i^l\right)$$
are basic sets of repelling type and saddle type respectively with $\Lambda_r^l\subset\widetilde\Lambda_r^l,$ $\Lambda_s^l\subset\widetilde\Lambda_s^l,$ which are topologically mixing on their natural extensions with topological entropy equals to $\log3.$
\end{proposition}
\begin{proof}
By Lemma \ref{le-triangle} and Proposition \ref{prop-exi-rep}, there exist $1/10>\alpha_1>0$ $,\epsilon_1,\delta_1>0$ and $l_1\geq1$ such that for all $l\geq l_1$ a map of the form
$$(z,w)\mapsto(h_1(z,w),q^l(w))$$
with $|h_1(z,w)-((1+\alpha_1)z+\epsilon_1r_i)|\leq\delta_1$ on $2\Db\times V_i^l$ has a blender of repelling type in $\Db\times\cup_{i=1}^3V^l_i.$  In particular, this holds for $h_1(z,w)=z+(\alpha_1z+\epsilon_1w)(w^3+1)/2$ for $l\geq l_1$ large enough since $w$ is close to $r_i$ on $V_i^l$ and thus $(w^3+1)/2$ is close to $1.$ Notice that this still holds for perturbations of $h_1$ of uniform size with respect to $l\geq l_0.$

We now explain why one can perturb such a map in order to obtain in addition a blender of saddle type. For what follows, we need that this blender is contained in a smaller set of the form $4^{-1}\Db\times\cup_{i=1}^3U_i^l.$ Observe that by a simple change of variables, the blender obtained by Proposition \ref{prop-exi-selle} can be chosen to be in $r\Db\times\cup_{i=1}^3U_i^l$ with $r>0$ arbitrarily small. Hence, by Lemma \ref{le-triangle} and Proposition \ref{prop-exi-selle} there exist $\alpha_2,\epsilon_2>0$ arbitrarily small, $\delta_2>0$ and $l_2\geq1$ such that for all $l\geq l_2$ a map of the form
$$(z,w)\mapsto(h_2(z,w),q^l(w))$$
with $|h_2(z,w)-((1-\alpha_2)z-\epsilon_2s_i)|\leq\delta_2$ on $2^{-1}\Db\times U_i^l$ has a blender of saddle type in $4^{-1}\Db\times\cup_{i=1}^3U_i^l.$ Since $w^3+1$ is close to $0$ on $U^l_i$ for $l\geq l_2$ large enough, this inequality holds for $h_2(z,w)=(z+(\alpha_1z+\epsilon_1w)(w^3+1)/2-\alpha_2z-\epsilon_2w).$ Moreover, as $\alpha_2, \epsilon_2>0$ can be chosen arbitrarily small, $h_2$ is an arbitrarily small perturbation of the function $h_1$ defined above. Hence, there exist $l_3\geq1$ and $\alpha_2, \epsilon_2>0$ such that for all $l\geq l_3$ the map $F^l(z,w):=(h_2(z,w),q^l(w))$ has two blenders $\Lambda^l_r$ and $\Lambda_s^l$ of repelling type and of saddle type respectively given by
$$\Lambda_r^l:=\bigcap_{n\geq0}F_l^{-n}\left(\overline{\cup_{i=1}^3H_i\times V_i^l}\right)\ \text{ and }\ \Lambda_s^l:=\bigcap_{n\geq0}F_l^{-n}\left(\overline{\cup_{i=1}^34^{-1}\Db\times U_i^l}\right).$$
If $R>0$ is large enough, it is easy to see that $\overline{R\Db}\subset h_2(R\Db\times\cup_{i=1}^3V^l_i),$ $h_2(\overline{R\Db\times\cup_{i=1}^3U^l_i})\subset R\Db,$ and that the sets $\widetilde\Lambda^l_r$ and $\widetilde\Lambda_s^l$ are hyperbolic, of repelling and saddle type respectively, and containing $\Lambda^l_r$ and $\Lambda_s^l$ respectively. Moreover, the dynamics on their natural extensions are topologically mixing with entropy $\log3$ since in both cases they are conjugated to the natural extension of $q^l$ on $\Lambda^l:=\cap_{n\geq0}q^{-nl}(\cup_{i=1}^3V_i^l)$ which is a two-sided full shift of three symbols.

It remains to prove the last point i.e. for each $\widehat x$ in the natural extension of $\Lambda_s^l$ we have $W^u_{\widehat x}\cap\Lambda^l_r\neq\varnothing.$ Recall that by Proposition \ref{prop-inter-robuste}, if a vertical graph in $H_i\times V^l_i$ is tangent to the cone field $C_{\delta_1}$ then it intersects $\Lambda^l_r.$ Since $q$ is uniformly expanding on the annulus $A:=\{w\in\Cb\,|\, 1/2<|w|<2\}$ for each $0<\delta<\delta_1$ there is $l\geq l_3$ such that $F_l$ contracts $C_{\delta}$ on $2\Db\times A.$ Hence, the local unstable manifold $W^u_{loc}(\widehat x)$ in $(4^{-1}\Db)\times U_i^l$ of a point $\widehat x=(x_i)_{i\leq0}\in\widehat\Lambda_s^l$ with $x_0\in\Lambda^l_s\cap (4^{-1}\Db)\times U_i^l$ is a vertical graph over $U_i^l$ which is tangent to $C_\delta.$ If $\delta>0$ is small enough, its image by $F_l$ contains a vertical graph in $H_j\times V^l_j$ for some $j\in\{1,2,3\}$ since we have assumed that $\overline{\cup_{j=1}^3V^l_j}\subset q^l(U_j^l).$ Here, we use that $\overline{4^{-1}\Db}\subset2^{-1}\Db\subset H_j$ and that $\delta<\min\{1/8\pi,\delta_1\}.$ Moreover, this graph is tangent to $C_\delta$ thus to $C_{\delta_1}$ and therefore by Proposition \ref{prop-inter-robuste}, $F_l(W^u_{loc}(\widehat x))\cap\Lambda_r^l\neq\varnothing,$ i.e. $W_{\widehat F_l(\widehat x)}^u\cap\Lambda_r^l\neq\varnothing.$ The result follows since $\widehat F_l(\widehat\Lambda_s^l)=\widehat\Lambda_s^l.$
\end{proof}

An important point about the above proof is that it is robust by small perturbations. Hence, if we fix $\alpha_1,\alpha_2,\epsilon_1,\epsilon_2,R>0$ and $l\geq l_0$ as above then there exists $\eta_0>0$ such that an endormophism $f$ of $\Pb^2$ with $\|F_l-f\|_{\infty,2R\Db\times2\Db}<\eta_0(2R)^{7l}$ has two blenders $\Lambda_r(f),$ $\Lambda_s(f)$ and two basic sets $\widetilde\Lambda_r(f),$ $\widetilde\Lambda_s(f)$ with the same properties than the ones state in Proposition \ref{prop-double-blender}. Moreover, by taking $\eta_0>0$ sufficiently small, Theorem \ref{th-motion} implies that $\widetilde\Lambda_r(f)$ moves holomorphically with $f.$ In particular, this applies to maps in a small neighborhood in $\Hc_d(\Pb^2)$ of the family of polynomial skew products $(f_\lambda)_{\lambda\in\Db^*}$ given by
$$f_\lambda(z,w):=(z+(\alpha_1z+\epsilon_1w)(w^3+1)/2-\alpha_2z-\epsilon_2w+\eta_0q^l(\lambda z),q^l(w)).$$
The following result is a reformulation of Theorem \ref{th-cycle}.
\begin{theorem}\label{th-cycle-bis}
If $\Omega\subset\Hc_d(\Pb^2)$ is a sufficiently small connected neighborhood of the family $(f_\lambda)_{\lambda\in\Db^*}$ then $\Omega\subset\bif(\Hc_d(\Pb^2))$ and for each $f\in\Omega$ the set $W_{\widetilde\Lambda_s(f)}^u$ is a Zariski open set, for each $\widehat x=(x_i)_{i\leq0}$ in the natural extension of $\widetilde\Lambda_s(f)$, $\overline{W^u_{\widehat x}}=\Pb^2$ and $W^s_{x_0}\subset\Jc_2(f).$
\end{theorem}
\begin{proof}
As the repelling periodic points are dense in $\widetilde\Lambda_r(f)$ and that for a skew product of $\Cb^2$ the closure of the repelling periodic points equals $\Jc_2(f),$ we have $\widetilde\Lambda_r(f)\subset\Jc_2(f)$ for all $f\in\Omega$ if $\Omega$ is a sufficiently small neighborhood of the family $(f_\lambda)_{\lambda\in\Db^*}$ (see \cite[Lemma 2.3]{dujardin-bif}).

Let $\widehat x$ and $\widehat y=(y_i)_{i\leq0}$ be in the natural extension of $\widetilde\Lambda_s(f)$ and $\Lambda_s(f)$ respectively. Let $\sigma$ be a small holomorphic disc transverse to $W^s_{x_0}$ near $x_0.$ As we have seen in the proof of Proposition \ref{prop-double-blender}, $W^u_{\widehat y}$ contains a vertical graph in $H_j\times V^l_j$ which is tangent to $C_{\delta_1}.$ Since $\widehat f$ acts transitively on the latter, it follows from the shadowing lemma (see \cite[Paper I, Theorem 2.4]{jonsson-phd}) that there exist a periodic point $\widehat z$ in the natural extension of $\widetilde\Lambda_s(f)$ and $n_0\geq0$ such that $\widehat z=(z_i)_{i\leq0}$ is close to $\widehat x$ and $\widehat f^{n_0}(\widehat z)$ is close to $\widehat y.$ By the continuity of the stable and unstable manifolds (see \cite[Paper I, Theorem 1.2]{jonsson-phd}), $\widehat z$ can be chosen such that $W^u_{\widehat f^{n_0}(\widehat z)}$ contains a vertical graph in $H_j\times V^l_j$ which is tangent to $C_{\delta_1}$ and $\sigma$ is transverse to $W^s_{z_0}$ near $z_0.$ By the inclination lemma (\cite[Paper I, Proposition 1.3]{jonsson-phd}) there exists $n_1\geq0$ such that $f^{n_1}(\sigma)$ also contains a vertical graph in $H_j\times V^l_j$ which is tangent to $C_{\delta_1}$ and thus by Proposition \ref{prop-inter-robuste} $f^{n_1}(\sigma)\cap\Lambda_r(f)\neq\varnothing.$ But $\Lambda_r(f)\subset\Jc_2(f)$ which is totally invariant so $\sigma\cap\Jc_2(f)\neq\varnothing.$ As $\Jc_2(f)$ is closed and $\sigma$ was an arbitrarily small disc transverse to the local stable manifold of $x_0$ we have that this local stable manifold is contained in $\Jc_2(f).$ Again, the total invariance of $\Jc_2(f)$ ensures that $W^s_{x_0}\subset\Jc_2(f).$

If $Z:=\cup_{i=1}^34^{-1}\Db\times U_i^l$ then by the construction of the blender of saddle type we have $\overline Z\subset f(Z).$ Hence, \cite[Paper I, Corollary 2.6]{jonsson-phd} implies $Z\subset W^u_{\widetilde\Lambda_s(f)}$ and so $\cup_{n\geq0}f^n(Z)\subset W^u_{\widetilde\Lambda_s(f)}.$ On the other hand, we have just seen that $Z$ intersects $\Jc_2(f)$ so by \cite{ds-allupoly}, if $\cali E(f)$ denotes the exceptional set of $f,$ which is a totally invariant critical analytic set, then $\Pb^2\setminus\mathcal E(f)\subset\cup_{n\geq0}f^n(Z)\subset W^u_{\widetilde\Lambda_s(f)}.$ But if $\Omega$ is sufficiently small then for each $f\in\Omega,$ $\widetilde\Lambda_s(f)$ is disjoint from the critical set so $\widetilde\Lambda_s(f)\cap\mathcal E(f)=\varnothing$ and therefore $W^u_{\widetilde\Lambda_s(f)}\cap\mathcal E(f)=\varnothing,$ i.e. $W^u_{\widetilde\Lambda_s(f)}$ is equal to the Zariski open set $\Pb^2\setminus\mathcal E(f).$ This also implies that $\overline{W_{\widehat x}^u}=\Pb^2$ for each $\widehat x$ in the natural extension of $\widetilde\Lambda_s(f)$ since classical results from hyperbolic dynamics imply that as $\widehat f$ is topologically mixing, we have $\overline{W^u_{\widetilde\Lambda_s(f)}}=\overline{W_{\widehat x}^u}.$

It remains to prove that $\Omega\subset\bif(\Hc_d(\Pb^2)).$ There are two steps. First we show that the postcritical set intersects $\Lambda_r(f)$ and then that this intersection is proper. The result will then follows from Proposition \ref{prop-ouvert-de-bif}. In both cases, we will use to family $(f_\lambda)_{\lambda\in\Db^*}.$

Let $\lambda\in\Db^*.$ Since $f_\lambda$ is a skew product, if $x$ is a periodic point in $\widetilde\Lambda_s(f_\lambda)$ then the stable manifold $W^s_x$ corresponds to an attractive basin in the fiber associated to $x.$ Hence, by a classical result of Fatou this basin must contain a critical point i.e. the critical set $C(f_\lambda)$ intersects $W^s_x.$ Moreover, since the periodic points are dense (and non-isolated) in $\widetilde\Lambda_s(f_\lambda)$ and that all stable manifolds are contained in horizontal fibers, there exists a periodic point $x$ such that the intersection between $C(f_\lambda)$ and $W^s_x$ is transverse. Hence, by the same arguments than above, the inclination lemma implies that the postcritical set of $f_\lambda$ intersects in a robust way the repelling blender $\Lambda_r(f_\lambda).$ Since this holds for all $\lambda\in\Db^*,$ if $\Omega\subset\Hc_d(\Pb^2)$ is a sufficiently small connected neighborhood of the family $(f_\lambda)_{\lambda\in\Db^*}$ then for each $f\in\Omega$ its postcritical set $P(f)$ intersects $\Lambda_r(f)\subset\widetilde\Lambda_r(f).$ Finally, $P(f)$ has to intersects properly $\widetilde\Lambda_r(f).$ Indeed, otherwise this intersection could be followed in the family $(f_\lambda)_{\lambda\in\Db^*}$ and thus in the conjugated family $(g_\lambda)_{\lambda\in\Db^*}$ where $g_\lambda:=\phi_\lambda\circ f_\lambda\circ\phi_\lambda^{-1}$ with $\phi_\lambda(z,w):=(\lambda z,w).$ At the limit, when $\lambda$ goes to $0,$ we should have an intersection between the postcritical set of the limit map and a limit value of the sets $\widetilde\Lambda_r(g_\lambda).$ However, as $g_\lambda(z,w)=(z(1+\alpha_1(w^3+1)/2-\alpha_2)+\lambda w(\epsilon_1(w^3+1)/2-\epsilon_2)+\lambda\eta_0q^l(z),q^l(w)),$ the limit map is $g_0(z,w):=(z(1+\alpha_1(w^3+1)/2-\alpha_2),q^l(w))$ and a simple computation shows that, under the assumption $\alpha_2<\alpha_1<1/10,$ its postcritical set doesn't intersect $\{(0,w)\in\Cb^2\,|\, 1/2<|w|<2\}$ which contains all the possible limit values of the sets $\widetilde\Lambda_r(g_\lambda)$ when $\lambda$ goes to $0.$
\end{proof}

\begin{remark}\label{rk-cycle}
The approach to obtain an open set of bifurcations in the proof above can be generalized in the following way. Assume that $f$ is a map with a blender of repelling type $\Lambda_r$ and a saddle periodic point $x$ such that its unstable manifold $W_x^u$ intersects $\Lambda_r$ and is tangent to the cone field of $\Lambda_r$ near this intersection. By a result of Robertson \cite{robertson}, the critical set intersects the stable manifold $W^s_x.$ If the intersection is transverse then as in the above proof the postcritical set must intersect $\Lambda_r$ in a robust way. And if this new intersection is proper in the sense of Proposition \ref{prop-ouvert-de-bif} then it gives an open set of bifurcations. However, these two assumptions on the intersections seem difficult to check in general.
\end{remark}

\section{Hénon perturbations}\label{sec-henon}
In this short section we consider small perturbations of polynomial automorphisms of $\Cb^2.$ It turns out that they can give rise to proper attracting sets with repelling points and to infinite sequences of nested attracting sets. The first point is a simple observation and the second one comes easily from the works of Gavosto \cite{gavosto} and Buzzard \cite{buzzard} on the Newhouse phenomenon in the holomorphic setting. It turns out that in \cite{dujardin-bif} Dujardin uses the same kind of examples to show that the Newhouse phenomenon is compatible with stability in the sense of \cite{bbd-bif}.

In \cite{buzzard}, Buzzard obtains a volume preserving polynomial automorphism $f_0$ of $\Cb^2$ with a persistent homoclinic tangency. Moreover, using \cite[Theorem 4.1]{gavosto} he proves that if $f_\epsilon$ is a small perturbation of $f_0$ which is volume decreasing near the tangency  then arbitrarily close to $f_\epsilon$ there exists a map with infinitely many attracting cycles. Actually, using exactly the same arguments but with a volume increasing perturbation $f_\epsilon,$ we obtain a map $f$ close to $f_\epsilon$ with infinitely many repelling cycles near the tangency. We will see that if $f$ is close enough to $f_0$ then all these cycles belong to a proper attracting set $A$ of $f$ and each time we remove to $A$ one of these cycles, we obtain a smaller attracting set. Notice that this set $A$ is the counterpart to (the closure of) the standard set $K^-(f_0)=\{(z,w)\in\Cb^2\,|\, (f_0^{-n}(z,w))_{n\geq0} \text{ is bounded}\}$ associated to $f_0,$ (see also Remark \ref{rk-henon}).

\begin{proof}[Proof of Theorem \ref{th-pertu}]
The maps with persistent homoclinic tangencies obtained by Buzzard are of the form $f_0:=F_3\circ F_2\circ F_1$ with $F_1(z,w)=(z+g_1(w),w),$ $F_2(z,w)=(z,w+g_2(z))$ and $F_3(z,w)=(cz,c^{-1}w)$ for some polynomials $g_1$ and $g_2$ of the same degree $d$ and  some $0<c<1.$ As observed by Dujardin in \cite{dujardin-bif}, $f_0$ can also be seen as the composition of two Hénon maps $h^\pm$ of the form $h^\pm(z,w)=(w,c^{\pm1}z+p^\pm(w))$ where $p^+$ and $p^-$ are two polynomials of degree $d.$ It is a standard fact about Hénon maps that if $R>0$ is large enough then the sets 
$$V^+:=\{(z,w)\in\Cb^2\,|\,|w|\geq\max(R,|z|)\}\ \text{ and }\ V^-:=\{(z,w)\in\Cb^2\,|\,|z|\geq\max(R,|w|)\}$$
satisfy $h^\pm(V^+)\subset V^+$ and $(h^\pm)^{-1}(V^-)\subset V^-.$ Indeed, with the same proof one can obtain that if $R>0$ is large enough then $(h^\pm)^{-1}(V^-)\subset W^-$ where $W^-:=\{(z,w)\in\C^2\,|\, |z|\geq2\max(R,|w|)\}.$ Moreover, seen as rational maps of $\Pb^2$ both $h^\pm$ have as unique indeterminacy point $I(h^\pm)=[1:0:0]$ and contract the line at infinity (minus $I(h^\pm)$) in a single point $[0:1:0].$ We deduce from these observations that $(h^\pm)(\Pb^2\setminus\overline{W^-})\subset\Pb^2\setminus\overline{V^-}.$ Hence, $U:=\Pb^2\setminus\overline{W^-}$ is a trapping region for $h^\pm$ and thus for $f_0$ and its small perturbations.

On the other hand, the homoclinic tangency of $f_0$ has to belong to the bidisc $V:=\{(z,w)\in\Cb^2\,|\, \max(|z|,|w|)<R\}.$ We choose a small perturbation $f_\epsilon$ of $f_0$ such that $f_\epsilon$ is an endomorphism of $\Pb^2$ which is volume increasing on $V,$ injective on $f^{-1}_\epsilon(V)\cap U$ and such that $\overline{f_\epsilon(U)}\subset U.$ For example, the map $f_\epsilon=h^-_\epsilon\circ h^+_\epsilon$ with $h_\epsilon^-(z,w)=(w+a_\epsilon z^d,c^{-1}z+p^-(w))$ and $h_\epsilon^+(z,w)=(w+a_\epsilon z^d,(c+\epsilon)z+p^+(w))$ where we first choose $\epsilon>0$ small and then $a_\epsilon\neq0$ very small with respect to $\epsilon$ is appropriate. Moreover, if $\epsilon$ is small enough then the results of \cite{buzzard} and \cite{gavosto} yield the existence of an endomorphism $f$ of $\Pb^2$ arbitrarily close to $f_\epsilon$ with infinitely many repelling cycles in $V\subset U.$ In particular, $U$ is also a trapping region for $f$ and $f$ is injective on $f^{-1}(V)\cap U.$ Therefore, $A:=\bigcap_{k\geq0}f^k(U)$ is a proper attracting set with infinitely many repelling cycles. Let $(C_n)_{n\geq0}$ denote the sequence of these repelling cycles in $A.$ The injectivity of $f$ on $f^{-1}(V)\cap U$ implies two facts. First, a repelling periodic point cannot be accumulated by other periodic points. In particular, there exists a family of disjoint neighborhoods $B_n\subset A\cap V$ of $C_n$ such that $\overline{B_n}\subset f(B_n).$ The second point implied by the injectivity of $f$ is that the open sets defined inductively by $U_0:=U$ and $U_{n+1}:=U_n\setminus\overline{B_n}$ are trapping regions. The associated attracting sets $A_n:=\bigcap_{k\geq0}f^k(U_n)$ satisfy $A_{n+1}\subsetneq A_n.$ Hence, $A_\infty=\bigcap_{n\geq0} A_n$ is a quasi-attractor. Moreover, $A_\infty$ cannot be an attracting set since any neighborhood $U_\infty$ of $A_\infty$ such that $f(U_\infty)\subset U_\infty$ must contains infinitely repelling cycles $C_n$ and thus $C_n\subset\bigcap_{k\geq0}f^k(U_\infty)\setminus A_\infty.$ 
\end{proof}
\begin{remark}\label{rk-uncountable}
Actually, the construction above leads to uncountably many different quasi-attractors. Let $X$ be an infinite subset of $\Nb.$ Then $A_X:=\cap_{n\in X}A_n$ is a quasi-attractor which is not an attracting set. And two different infinite subsets $X$ and $X'$ give different quasi-attractors $A_X$ and $A_{X'}.$
\end{remark}
\begin{remark}
Notice that in the situation above, the results of Dinh \cite{d-attractor} on attracting sets and attracting currents do not apply directly to the sets $A_n$ with $n\geq1$ and a fortiori to the quasi-attractor $A_\infty.$ On the other hand, we obtain in \cite{t-attractor} the existence of an attracting current $\tau$ supported on $A_\infty.$
\end{remark}

\begin{remark}\label{rk-henon}
Complex Hénon mappings have been extensively studied (see e.g. \cite{bs-1} and \cite{bls-4}). In particular, one can associate to such a map filled Julia sets $K^+,$ $K^-,$ $K=K^+\cap K^-$ and Julia sets $J^\pm=\partial K^\pm,$ $J=J^+\cap J^-.$ Moreover, there exist two positive closed $(1,1)$-currents $\mu^+$ and $\mu^-$ with $\supp(\mu^\pm)=J^\pm.$ The measure $\mu:=\mu^+\wedge\mu^-$ is mixing, hyperbolic, supported in $J,$ the saddle periodic points equidistribute toward it and it is the unique measure of maximal entropy $\log(d).$ As we have seen, we can perturb such a map in order to have an endomorphism $f$ of $\Pb^2$ which possesses an attracting set $A.$ The map $f$ has a Green current $T$ whose support is exactly the Julia set $\Jc_1(f)$ and there exists an attracting current $\tau$ supported in $A.$ There is a strong analogy between these objects. The currents $T$ and $\tau$ correspond respectively to $\mu^+$ and $\mu^-,$ the sets $J^+$ and $K^-$ correspond to $\Jc_1(f)$ and $A.$ Moreover, in this situation the combination of the results of \cite{d-attractor}, \cite{da-ta} and \cite{daurat-2} gives that the measure $\nu:=T\wedge\tau$ is mixing, hyperbolic, the saddle periodic points equidistribute toward it and it is the unique measure supported in $A$ of maximal entropy $\log(d)$ i.e. it can be seen as the continuation of $\mu$ for $f.$
\end{remark}


\bibliographystyle{alpha} 
\newcommand{\etalchar}[1]{$^{#1}$}

\noindent
Universit\'e de Bourgogne Franche-Comt\'e, IMB, UMR CNRS 5584, 21078 Dijon Cedex, France. {\tt e-mail:johan.taflin$@$u-bourgogne.fr}

\end{document}